\documentclass[12pt,reqno]{amsart}
\usepackage[a4paper]{geometry}
\usepackage{amsthm,hyperref,lmodern,mathrsfs,amssymb,amsmath,amsfonts}

\usepackage[utf8]{inputenc}

\usepackage{color}

\usepackage{xcolor}

\definecolor{boubelcolor}{rgb}{.65,0.05,0}

\DeclareUnicodeCharacter{2260}{\colorlet{memoireboubel}{.}\color{boubelcolor}} 
\DeclareUnicodeCharacter{00B1}{\color{memoireboubel}{}} 
\DeclareUnicodeCharacter{00A3}{\colorlet{memoirejuillet}{.}\color{blue}} 
\DeclareUnicodeCharacter{00A7}{\color{memoirejuillet}{}} 

\newtheorem{theo}{Theorem}[section]
\newtheorem{lemme}[theo]{Lemma}
\newtheorem{prop}[theo]{Proposition}

\newtheorem{defi}[theo]{Definition}
\newtheorem{remark}[theo]{Remark}

\newcommand{\h}{\mathrm h}
\newcommand{\He}{\mathbb H}
\newcommand{\R}{\mathbb R}

\newcommand{\F}{\mathcal F}
\newcommand{\W}{\mathcal W}
\newcommand{\eps}{\varepsilon}

\newcommand{\B}{\mathbf B}

\newcommand{\E}{\mathbb E}

\newcommand{\tr}{\mathrm{tr}}
\newcommand{\dil}{\mathrm{dil}}
\renewcommand{\P}{\mathbb P}
\newcommand{\law}{\mathrm{Law}}
\newcommand{\trans}{\mathrm{trans}}
\newcommand{\rot}{\mathrm{rot}}

\title[$L^p$ coupling of two Brownian motions and their L\'evy area]{Couplings in $L^p$ distance of two Brownian motions and their L\'evy area}

\begin{document}
\author{Michel Bonnefont}
\address{Institut de Math\'ematiques de Bordeaux, Universit\'e
 de Bordeaux, $351$, cours de la Lib\'eration
\\$33405$ Talence cedex, France}
\email{michel.bonnefont@math.u-bordeaux.fr}
\author{Nicolas Juillet}
\address{Institut de Recherche Math\'ematique Avanc\'ee, Universit\'e
de Strasbourg,  7 rue Ren\'e-Descartes
 67084 Strasbourg cedex France}
\email{nicolas.juillet@math.unistra.fr}
\date{Version of \today}

\begin{abstract}
We study co-adapted couplings of (canonical hypoelliptic) diffusions on the (subRiemannian) Heisenberg group, that we call (Heisenberg) Brownian motions and are the joint laws of a planar Brownian motion with its L\'evy area. We show that contrary to the situation observed on Riemannian manifolds of non-negative Ricci curvature, for any co-adapted coupling, two Heisenberg Brownian motions starting at two given points can not stay at bounded distance for all time $t\geq 0$. Actually, we prove the stronger result  that they can not stay bounded in $L^p$ for $p\geq 2$.

We also prove two positive results. We first study the coupling by reflection and show that it stays bounded in $L^p$ for $0\leq p <1$. 
Secondly, we construct an explicit static (and in particular non co-adapted) coupling between the laws of two Brownian motions, which provides  $L^1$-Wasserstein control uniformly in time.

Finally, we explain how the results generalise to the Heisenberg groups of higher dimension.
\end{abstract}

\maketitle

\section{Introduction}
\subsection{$L^\infty$ control}
The motivation for this paper is a question, concerning heat diffusion on the Heisenberg group, that is implicitly raised by Kuwada in \cite[Remark 4.4]{Ku}, and that we reproduce at page \pageref{la_question} after Theorem \ref{thm:Li-K}. Before we reach this question let us start with some background and a few definitions. All remaining material will be introduced later in the paper. In the literature, $L^\infty$-Wasserstein control for a diffusion has been used to deduce  $L^1$-gradient estimates of its associated semigroup (see for instance \cite{W} and the references therein). Kuwada extends this result to $L^p$-Wasserstein control and $L^q$-gradient estimates for all $p,q\geq 1$ with $\frac {1}{p} + \frac{1}{q}=1$ and, using Kantorovich  duality, proves that, conversely, $L^q$-gradient estimates allow one to obtain $L^p$-Wasserstein control for the diffusion.

We recall that, on a metric space $(M,d)$, for $p\in (0,\infty]$, the $L^p$-Wasserstein distance between two probability measures $\mu$ and $\nu$ is given by
\begin{equation}\label{eq:Wasserstein}
\W_p(\mu,\nu)=\left( \inf_{\pi \in \Pi(\mu,\nu)}  \iint d (x,y)^p d\pi(x,y) \right)^{1/p}=\inf_{X\sim \mu,\,Y\sim \nu}\|d_\He(X,Y)\|_p. 
\end{equation}
Here $\Pi(\mu,\nu)$ is the set of probability measures on $M \times M$ with marginals $\mu$  and $\nu$. For $p=\infty$ the first expression is replaced by the essential supremum of $d$. Note that $ \W_p$ is a distance only for $p\geq 1$. For  $0< p<1$, it  is only   a quasidistance, in the sense that the triangle inequality only holds up to a multiplicative constant. Using H\"older inequality, it is clear that $\W_p(\mu, \nu)  \leq \W_q(\mu,\nu)$ if $0<p\leq q$.

On the Heisenberg group $\He$, the following $L^1$-gradient bound was established by H.Q. Li \cite{Li} (see also \cite{BBBC}) generalising \cite{DM}
\[
\forall f\in \mathcal{C}_c^{\infty}(\He_1),\forall t\geq0,\forall a\in \He, |\nabla_\h P_t f (a)| \leq C P_t(|\nabla_\h f|)(a), 
\]
where $C>1$ is constant, $P_t$ denotes the heat semigroup associated to half the sub-Laplacian and $\nabla_\h$ the horizontal gradient (see Section \ref{sec:H} for the definitions). Consequently, Kuwada's result implies that the heat diffusion of the Heisenberg group possesses a $L^\infty$-Wasserstein control for its diffusion:

\begin{theo}[H.Q. Li, Kuwada]\label{thm:Li-K}
 There exists $C>0$ such that for every $t\geq 0$  and $a,a'\in \He$
\begin{equation}\label{eq:LiK}
 \mathcal{W}_\infty(\mu^a_t,\mu^{a'}_t)\leq Cd_\He(a,a'),
\end{equation}
where $\mu^a_t=\law(\B^a_t)$ and $\mu^{a'}_t=\law(\B^{a'}_t)$ and $(\B_s^a)_{s\geq 0},  (\B_s^{a'})_{s\geq 0}$ are two Heisenberg Brownian motions, starting respectively in $a,a'$. Moreover,
\begin{equation}\label{eq:LiK2}
 \mathcal{W}_p(\mu^a_t,\mu^{a'}_t)\leq Cd_\He(a,a')
\end{equation}
holds for every $p<\infty$.
\end{theo}
In other words, for each $a,a'\in \He$ and each $t\geq 0$, there exists a coupling $(\B_s^a, \B_s^{a'})_{s\geq 0}$ of  the two Heisenberg Brownian motions  such that 
 \begin{equation} \label{eq:coupling-t}
 d_{\He}(\B_t^a, \B_t^{a'}) \leq C d_{\He} (a,a')\quad\text{almost surely}.
 \end{equation}
Please note: Firstly the time $t\geq 0$ is fixed; secondly $\B_t^a$ and $\B_t^{a'}$ are conveniently defined on the same probability space (and the remaining random variables $(\B^a_s)_{s\neq t}$ and $(\B^{a'}_s)_{s\neq t}$ of the Heisenberg Brownian motions are defined without paying attention to their correlation). Kuwada's problem is precisely on inverting the quantifiers $\forall$ and $\exists$, namely, he asks whether it is possible to define a coupling of the two Heisenberg Brownian motions $(\B_t^a)_{t\geq 0}$ and $(\B_t^{a'})_{t\geq 0}$ such that \eqref{eq:coupling-t} holds for all $t\geq 0$.\label{la_question}

In this paper we answer negatively and show that $\eqref{eq:coupling-t}$ can not hold for all $t\geq 0$ for \emph{co-adapted} couplings (see Definition \ref{def:co-adapted}), probably the most usual couplings in the literature for our type of problem, (see, e.g. \cite{BaCK, K,Ke_ecp,K2,Cra91,Re_ejp,KuS,PaPo}).
Informally, a coupling  of two processes $(\B_t)$ and $(\B_t')$ is said  co-adapted if the interaction in the coupling only depends on the common past of the process $(\B_t, \B'_t)$. See Definition \ref{def:co-adapted} fo the rigorous definition.

 Our results hold for the Heisenberg groups of higher dimension, as explained in Section \ref{sec:gene}, but we only prove them thoroughly in the first Heisenberg group where all the significative ideas are present and the notation is lighter.

\begin{theo} \label{thm:rapport-Winfty}
For every $T>0$ and every $C>0$ there exists two points $a,a'\in \He$ with $a\neq a'$ such that for every co-adapted coupling $(\B_t^a, \B_t^{a'})_{0\leq t\leq T}$, there exists $t\leq T$ such that
\[ \mathrm{essup}_{(\Omega,\P)} \, d_{\He}(\B_t^a, \B_t^{a'})> Cd_\He(a,a').\]
\end{theo}

\begin{remark} \label{rem:rapport-Winfty}
Another way to state Theorem \ref{thm:rapport-Winfty} is as follows: Let $T>0$, then 
\begin{align*}
 \sup_{a\neq a'\in\He} \inf_{   \mathcal A_ T ^{(a,a')} } \sup_{0\leq t\leq T}    \frac{ \mathrm{essup}_{(\Omega,\P)}   d_{\He}(\B_t^a, \B_t^{a'}) }{ d_{\He} (a,a')}=+\infty
\end{align*}
where $ \mathcal A_ T ^{(a,a')}$  denotes the set of all co-adapted couplings $(\B_t^a, \B_t^{a'})_{0\leq t\leq T}$ of two Heisenberg Brownian motions starting respectively in $a$ and $a'$.
 \end{remark}

The proof will be based on the following result and the use of homogeneous dilations (defined in Section \ref{sec:H}):

\begin{theo}\label{thm:Winfty}\label{them:LiKu}
 Let  $(\B_t)_t$ and  $(\B'_t)_t$  be any two co-adapted Heisenberg Brownian motions starting respectively in $a=(x,y,z)$ and $a'=(x',y',z')$ with $(x'-x)^2 + (y'-y)^2>0$. Then, for every $C>0$,
 \[
  \P\left(\forall t \geq 0,\,d_\He ( \B_t,\B'_t) \leq C\right)\neq1.  
 \]
\end{theo}

\subsection{Comparison with the Riemannian case}
These results show a significative difference with the Riemannian case. Indeed, on a Riemannian manifold $M$,  it is well known (see e.g. \cite {W} and \cite{VRS}) that if the Ricci curvature is bounded from below by $k\in\R$, there exists a Markovian coupling of two Brownian motions such that almost surely 
 \begin{equation} \label{eq:coupling-t-riem}
 d(\B_t^a, \B_t^{a'}) \leq e^{-(k/2)t} d (a,a')\quad \textrm{ for all } t\geq 0,\,a\in M,\,a'\in M.
 \end{equation}
Here we call Brownian motions the diffusion processes starting at $a$ and $a'$ respectively, having generator half the Laplace--Beltrami operator. We make clear that Markovian coupling is a type of co-adapted coupling. Note moreover that the motivation for proving \eqref{eq:coupling-t-riem} is exactly to provide estimates on the heat semi-group (see, e.g. \cite{Cra91,Cra92}), so that the historical $L^p$-Wasserstein controls have been established for co-adaptive processes whereas $L^p$-Wasserstein controls at fixed time may first appear unusual from a stochastic perspective.

We note further that
\begin{itemize}
\item the Heisenberg group can be thought as the first sub-elliptic model space of curvature $0$ (e.g. \cite{Mg}) but, its behaviour with respect to couplings of co-adapted Brownian motions is therefore completely different from the case of Riemannian manifolds with curvature bounded from below by $k=0$.
\item the Heisenberg group is also classically presented as the limit space for a sequence of Riemannian metrics on the Lie group, the optimal lower bound on the Ricci curvature of which tends to $-\infty$. On this topic see \cite{Ju_imrn,BG}. This fact is coherent with the interpretation of Theorem \ref{thm:Winfty} as a special case of \eqref{eq:coupling-t-riem} where the best bound for the $L^\infty$ control is $C=e^{-kt}$ with $k=-\infty$: There is no possible control for $t>0$.
\end{itemize}

\subsection{$L^p$ control for $0<p<\infty$}
 To go further,  given  two  diffusion processes  $(\B_t)_{t\geq 0}$ and $(\B'_t)_{t\geq 0}$ on a metric space $(M,d)$, we shall consider the function
\[
t\in [0,\infty) \to \E\left[d(\B_t, \B'_t)^p \right]^\frac{1}{p}\in [0,\infty].
\]
and try to bound it from above uniformly in time for some well-chosen co-adapted coupling.
If we  denote by $\mu_t$ and $\nu_t$ the law of the processes  $(\B_t)_{t\geq 0}$ and $(\B'_t)_{t\geq 0}$, we clearly have for each $t\geq 0$:
\[
 \E\left[d(\B_t, \B'_t)^p \right]^\frac{1}{p} \geq  \W_p (\mu_t,\nu_t).
\]

On the Heisenberg group, we will prove the result stronger than Theorem \ref{thm:rapport-Winfty} that any co-adapted coupling $(\B_t)_t$ and  $(\B'_t)_t$ of Brownian motions do not stay bounded in $L^{2}$:

\begin{theo} \label{thm:rapport-W2}
Let $p\geq 2$.  For every $T>0$ and every $C>0$ there exists two points $a,a'\in \He$ with $a\neq a'$ such that for every co-adapted coupling $(\B_t^a, \B_t^{a'})_{0\leq t\leq T}$, there exists $t\leq T$ such that
\[ 
\E \left[  d_{\He}^p(\B_t^a, \B_t^{a'})\right]> Cd_\He(a,a')^p.\]
\end{theo}

\begin{remark} \label{rem:rapport-W2}
Equivalently Theorem \ref{thm:rapport-W2} can be stated as follows: Let $p\geq 2$.   Let $T>0$, then 
\begin{align*}
 \sup_{a\neq a'\in\He}   \inf_{   \mathcal A_ T ^{(a,a')} }    \sup_{0\leq t\leq T}  \frac{ \E[ d_{\He}^p(\B_t^a, \B_t^{a'}) ]^\frac{1}{p}
} { d_{\He} (a,a')}=+\infty
\end{align*}
where $ \mathcal A_ T ^{(a,a')}$ denotes the set of all co-adapted coupling $(\B_t^a, \B_t^{a'})_{0\leq t\leq T}$ of two Heisenberg Brownian motions starting respectively in $a$ and $a'$.
\end{remark}

As for $p=\infty$, the proof will be based on the following result and the use of dilations.
\begin{theo}\label{thm:W2}
 Let  $(\B_t)_{t\geq 0}$ and  $(\B'_t)_{t\geq 0}$  be any two co-adapted Heisenberg Brownian motion starting respectively in $a=(x,y,z)$ and $a'=(x',y',z')$ such that $(x'-x)^2 + (y'-y)^2>0$. Then, 
 \[
   \limsup_{t\to +\infty} \E\left[d_\He ( \B_t,\B'_t)^2 \right] \to +\infty.  
 \]
\end{theo}

\subsection{Two positive results}
To complete the picture, we provide two positive results. We first  show that the coupling by reflection on the Heisenberg group stays bounded in  $L^p$ for $0< p <1$. We recall that for  $0< p<1$, the quantity  $ \E\left[d^p (\B_t, \B'_t) \right]^\frac{1}{p} $ is not a distance, but only a quasidistance, in the sense that the triangle inequality only holds up to a multiplicative constant.

\begin{theo}\label{thm:refl}
 Let $(\B_t)_{t\geq 0}$ and  $(\B'_t)_{t\geq 0}$ be a coupling by reflection of two Heisenberg Brownian motions starting in $(x,y,z)$ and $(x',y',z')$. Then, for every $p\in (0,1)$, 
 \begin{equation}\label{eq:sup-alpha}
   \sup_{t\geq 0} \E\left[d_\He ( \B_t,\B'_t)^p \right] < +\infty.  
 \end{equation}

 Moreover, for the coupling by reflection,  for every $p\in (0,1)$, we also have:
 \begin{equation}\label{eq:sup-rapport}
 \sup_{a\neq a'\in\He} \sup_{t\geq 0}  \frac{ \E[ d_{\He}(\B_t^a, \B_t^{a'})^p ]}{ d_{\He} (a,a')^p}<+\infty
\end{equation} 
 \end{theo}
 
Unfortunately, the above result is false for the reflection coupling for $p\geq 1$ (as a close look at  Proposition \ref{pro:assymp} shows).

In the general context of co-adapted coupling, we were not able  to obtain any results for $p \in [1,2)$: we ignore whether  there exist co-adapted couplings satisfying  \eqref{eq:sup-alpha} or  \eqref{eq:sup-rapport}, or not, for $p\in [1,2)$. One difficulty in this study is to obtain estimates for the expectation of nonnegative (nonconvex) functionals of martingales as typically $x\mapsto |x|^{1/2}$, see Remark \ref{rem:concave}.

The next remark recalls the situation of $L^p$-Wasserstein control in the case of Riemannian manifolds.
\begin{remark}
On a Riemannian manifold $M$, for $p\geq 1$, the $L^p$ version of the $L^\infty$ control \eqref{eq:LiK}, namely 
\[
 \mathcal{W}_p(\mu^a_t,\mu^{a'}_t)\leq e^{-kt/2}d(a,a')\textrm{ for all } t\geq 0,\,a\in M,\,a'\in M;
 \]
is satisfied  if and only if the Ricci curvature is bounded from below by $k$ (see \cite{VRS} and \cite[Remark 2.3]{Ku}). Therefore all the above $L^p$-Wasserstein controls, for $p\in [1,\infty]$ are equivalent and only depend on the Ricci curvature lower bound.
Moreover in this situation, as said before, one has the existence of some  appropriate Markovian  coupling such that almost surely:
\[
d( B_t^a, B_t^{a'}) \leq  e^{-kt/2}d(a,a') \textrm{ for all }t\geq 0,\,a\in M,\,a'\in M.
\] 
 \end{remark}

We now turn to the second positive result. We propose an explicit coupling of the laws $\mu^a_t=\law(\B^a_t)$ and $\mu^{a'}_t=\law(\B^{a'}_t)$. This coupling is not at all dynamical and is made at a given fixed time $t$. This coupling has thus no interpretation in terms of co-adapted coupling. It provides a new proof of the case $p=1$ in Theorem \ref{thm:Li-K}. This is the weakest result in the spectrum of $L^p$-Wasserstein controls ; the strongest result is for $p=\infty$. However, we stress that, apparently, our static coupling provides the first direct proof (that is, not obtained by duality) of the case $p=1$.

\begin{theo}\label{thm:static}
There exists $C>0$ such that for every $t\geq 0$ and $a=(x,y,z), a'=(x',y',z')\in \He$, there is a random vector $(X,Y,Z,X',Y',Z')$ of marginals $\mu^a_t=\law(\B^a_t)=\law(X,Y,Z)$ and $\mu^{a'}_t=\law(\B^{a'}_t)=\law(X',Y',Z')$ such that
\begin{align*}
&\mathcal{W}_1(\mu^a_t,\mu^{a'}_t)\leq \E(d_\He((X,Y,Z),(X',Y',Z')))\leq Cd_\He(a,a')\\
\intertext{and}
&\left\{
\begin{aligned}
X'=X+(x'-x)\quad\text{almost surely,}\\
Y'=Y+(y'-y)\quad\text{almost surely.}
\end{aligned}
\right.
\end{align*}
\end{theo}
The idea of the coupling is the following. The goal is to construct a random vector $((X,Y,Z);(X',Y',Z'))$ of marginals $\mu^a_t$ and $\mu^{a'}_t$. First we  perform a coupling of the horizontal parts $(X,Y)$ and $(X',Y')$ of the two Brownian laws by a simple translation. It appears that the  conditional laws of the last coordinates $\mathcal L(Z| (X,Y))$ and  $\mathcal L(Z'| (X',Y'))$  differ also only  by a translation;  but which depends on the value of $(X,Y)$. Eventually, we use a coupling of the last coordinates  which is well adapted to optimal transport for the cost $(z,z')\mapsto\sqrt{ |z-z' |}$ on the real line, and is better than the simple translation.
Note that our proof requires at the end an analytic estimate on the heat kernel, see  \eqref{eq:dzp1}.

We mention the interesting recent work by S. Banerjee, M. Gordina and P. Mariano \cite{BGM} where the authors also use non co-adapted couplings to study the decay in total variation for the laws of Heisenberg Brownian motions and obtain gradient estimates for harmonic functions. This work and our work seems to deliver a common message namely that co-adapted couplings are not the unique relevant couplings, what concerns obtaining gradient estimates.

\medskip
The paper is organised as follows. In Section \ref{sec:H}, we recall the  notion of co-adapted coupling and describe quickly the geometry of the Heisenberg group, its associated Brownian motions and their coupling. We also discuss some classical couplings. The proofs of the main theorems on the non-existence of co-adapted Heisenberg Brownian motions which stay at bounded distance are given Section \ref{sec:sec}.  The reflection coupling on $\He$ is studied in Section \ref{sec:refl}. The construction of the static coupling  between the Brownian laws is given in Section \ref{sec:bonus}.  The results are then generalised to the Heisenberg groups of higher dimension in the final section.

\section{ Co-adapted couplings on the Heisenberg group}
\subsection{The Heisenberg group}\label{sec:H}

The Heisenberg group can be identified with $\R^3$ equipped with the law:
\[
  (x,y,z)\cdot(x',y',z')= \left(x+x',y+y', z+z'+\frac{1}{2}(xy'-yx')\right).
 \]

 The left invariant vector fields are given by 
 \[
 \left\{ \begin{array} {l}
 X (f) (x,y,z)= \frac{d}{dt} _{|t=0} f( (x,y,z) \cdot (t,0,0 ) ) = \left( \partial_x - \frac{y}{2} \partial z \right) f(x,y,z) \\
  Y (f) (x,y,z)= \frac{d}{dt} _{|t=0} f( (x,y,z) \cdot (0,t,0 ) ) = \left( \partial_y + \frac{x}{2} \partial z \right) f(x,y,z)\\
Z (f) (x,y,z)= \frac{d}{dt} _{|t=0} f( (x,y,z) \cdot (0,0,t ) ) = \partial_z f(x,y,z).
\end{array} \right.
 \]
 
 Note that $[X,Y]=Z$ and that $Z$ commutes with $X$ and $Y$.
 
 We are interested in half the sub-Laplacian $L= \frac{1}{2} (X^2+Y^2)$. This is a diffusion operator that satisfies the H\"ormander bracket condition and thus the associated heat semigroup $P_t=e^{tL}$ admits a $\mathcal C^\infty$ positive kernel $p_t$.
 
 From a probabilistic point of view, $L$ is the generator of the following stochastic process starting in $(x,y,z)$:
 \[
 \B^{(x,y,z)}_t:= (x,y,z) \cdot \left( B^1_t,B^2_t, \frac{1}{2}\left( \int_0^t B^1_s dB^2_s -  \int_0^t B^2_s dB^1_s\right) \right)
 \]
 where $(B^1_t)_{t\geq 0}$ and  $(B^2_t)_{t\geq 0}$ are two standard independent 1-dimensional Brownian motions. The quantity $ \int_0^t B^1_s dB^2_s -  \int_0^t B^2_s dB^1_s$ that we denote by $A_t$ is one of the first stochastic integrals ever considered. It is the L\'evy area of the 2-dimensional Brownian motion $(B_t)_{t\geq 0} :=(B^1_t, B^2_t)_{t\geq 0}$.
  
It is easily seen that $(\B_t)_{t\geq 0} $ is a continuous process with independent and stationary increments. We simply call it the Heisenberg Brownian motion.

The sub-Laplacian $L$ is strongly related to the following  subRiemmanian distance (also called Carnot-Carath\'eodory) on $\He$:
\[
d_{\He} (a,a')= \inf_{\gamma}  \int_0 ^1  | \dot \gamma(t)  |_\h dt
\]
where $\gamma$ ranges over the horizontal curves connecting $\gamma(0)=a$ and $\gamma(1)=a'$. We remind the reader of the fact that a curve is said horizontal if it is absolutely continuous and $\dot \gamma (t) \in \textrm{Vect} ( X(\gamma(t)), Y(\gamma(t)))$ almost surely holds. The horizontal norm $|\cdot |_\h$ is a Euclidean norm on $\textrm{Vect} (X,Y)$ obtained by asserting that $(X,Y)$ is an orthonormal basis of $\textrm{Vect} (X(a),Y(a))$ at each point $a\in \He$. Finally the horizontal gradient $\nabla_\h f$ is $(Xf) X+(Yf)Y$.

 The Heisenberg group admits homogeneous dilations adapted both to the distance and the group structure. They are given by 
  \[
  \dil_\lambda (x,y,z)= (\lambda x, \lambda  y, \lambda^2 z)
  \]
for $\lambda>0$. They satisfy $d_\He(\dil_\lambda(a),\dil_\lambda(a'))=\lambda d_\He(a,a')$ and, in law:
 
\[
\dil _{\frac{1}{\sqrt t}} \left( B^1_t,B^2_t, \frac{1}{2}\left( \int_0^t B^1_s dB^2_s -  \int_0^t B^2_s dB^1_s\right) \right) \overset{\mathrm{Law}}{=}   \left( B^1_1,B^2_1, \frac{1}{2}\left( \int_0^1 B^1_s dB^2_s -  \int_0^1 B^2_s dB^1_s\right) \right). 
\]

The distance is clearly left-invariant so that $\trans_p:q\in \He\mapsto p.q$ is an isometry for every $p\in \He$. In particular
\[
d_{\He}(a,a)= d_{\He} (e, a^{-1} a')
\]
with $e=(0,0,0)$. Another isometry is the rotation $\rot_\theta:(x+iy,z)\in \mathbb{C}\times \R\equiv \He \mapsto (\mathrm{e}^{i\theta}(x+iy),z)$, for every $\theta\in \R$. Since  the explicit expression of $d_\He$ is not so easy, it is often simpler to work with a homogenous quasinorm (still in the sense that the triangle inequality only holds up to a multiplicative constant). We will use 
\[
H: a=(x,y,z)\in \He\mapsto\sqrt {x^2+y^2 + |z |}\in \R ,
\]
and the attached homogeneous quasidistance $d_H(a,a')=H(a^{-1}a')$. It satisfies 
\begin{equation}\label{eq:dist-eq}
c^{-1} d_H(a,a') \leq d_{\He}(a,a') \leq c d_H(a,a')
\end{equation}
for some constant $c>1$. We finally mention $d_\He((0,0,0),(x,y,0))=\sqrt{x^2+y^2}$ and $d_\He((x,y,z),(x,y,z+h))=2\sqrt{\pi |h|}$. 

\subsection{Co-adapted couplings}\label{sec:co-ad}

We first recall  the notion of  \emph{co-adapted} coupling of two processes. Indeed, in this study, we only  want to consider   couplings built solely knowing the  past of the two  processes. The definition below is taken from \cite[Definition 1.1.]{K2}.

\begin{defi}\label{def:co-adapted}
Given two continuous-time Markov processes $(X^{(1)}_t)_{t\geq0}$ , $(X^{(2)}_t)_{t \geq 0}$, we say that $(\tilde X^{(1)}_t,  \tilde X^{(2)}_t)_{t \geq 0}$ is a \emph{co-adapted coupling} of $(X^{(1)}_t)_{t\geq0}$ and $(X^{(2)}_t)_{t \geq 0}$ if $\tilde X^{(1)}$ and  $ \tilde X^{(2)}$ are defined on the same filtered probability space $(\Omega,  (\mathcal F_t)_{t\geq 0}, \mathbb P)$, satisfy $\law(X^{(i)}_t)_{t\geq 0}=\law(\tilde X^{(i)}_t)_{t\geq 0}$ for $i=1,2$, and 
\[
\tilde P_t^{(i)}f:z\mapsto\E\left[ f(\tilde X^{(i)}_{t+s}) |\, \mathcal F_s , \tilde X^{(i)}_s=z \right]
\]
equals
\[
P_t^{(i)}f:z\mapsto\E\left[ f(X^{(i)}_{t+s})|\,X^{(i)}_s=z \right],\quad\law(X^{(i)}_s)\text{-almost surely}
\]
 for $i=1,2$, for each bounded measurable function $f$, each $z$, each $s,t\geq 0$.
 \end{defi}
 
 If we moreover assume that the full process $(\tilde X^{(1)}_t,  \tilde X^{(2)}_t)_{t \geq 0}$ is Markovian, we say that the co-adapted coupling is Markovian.

 The next lemma describes more explicitly co-adapted couplings in the case of Brownian motion in $\R^2$ (see \cite[Lemma 6]{Ke_ecp}).
 
 \begin{lemme}\label{lem:gen-coupling}
Let $(B_t)_t$ and $(B'_t)_t$ be two co-adapted Brownian motions on $\R^2\times \R^2$  defined on some filtered probability space  $(\Omega, (\F_t)_{t\geq 0}, \P )$. Then, enriching the filtration if necessary,  there exists a Brownian motion $(\hat B_t)_{t\geq 0}$ defined on the same filtration $(\mathcal F_t)_{t\geq 0}$ and independent of $(B_t)_{t\geq 0} $ such that
 \begin{equation}\label{eq:co-adapted}
  d B'(t)= J(t) d B_t + \hat J(t) d \hat B_t
 \end{equation}
 where $(J_t)_{t\geq0}=((J_t^{i,j})_{1\leq i,j\leq 2})_{t\geq 0}$ and $\hat J$ are matrices satisfying
\begin{equation}\label{eq:co-BM}
  J J^T + \hat J \hat J^T= I_2
\end{equation}
and $J(t), \hat J (t) \in \F_t$.
\end{lemme}

In the following $\|\cdot\|$ may denote the operator norm of a matrix attached to the Euclidean norm, or the Euclidean norm of a vector.

\begin{lemme}\label{lem:deuxdeux}
Let $J$ be a $2\times2$ real matrix $J=\left(\begin{smallmatrix}a&b\\c&d\end{smallmatrix}\right)$. Then
$$0\leq J^TJ\leq I_2  \Longleftrightarrow \|J\|\leq 1 \Longleftrightarrow 0\leq JJ^T\leq I_2,$$
where $\leq$ is the ordering of symmetric matrices.
In particular
\begin{itemize}
\item $a^2+b^2$, $a^2+c^2$, $c^2+d^2$ and $b^2+d^2$ are smaller or equal to $1$,
\item all the four entries of $J$ are in $[-1,1]$.
\end{itemize}
\end{lemme}
\begin{proof}
Let $\mathbb{S}^1=\{(\cos(\theta),\sin(\theta))\in \R^2:\,\theta\in \R\}$ be the Euclidean sphere of $\R^2$ and $Q:x\in \mathbb{S}^1\mapsto (x,J^T J x)=\|Jx\|^2$. Therefore, $Q$ is bounded by $1$ if and only if $\|J x\|\leq1$, for all $x\in \mathbb{S}^1$. The bound $0\leq J^T J$ is trivially satisfied. The proof is completed by $\|J\|=\sup_{x,y\in \mathbb{S}^1}(Jx,y)=\|J^T\|$.
\end{proof}

\begin{remark}
A necessary and sufficient condition can be found considering $\lambda$, the greatest eigenvalue of $J^TJ$. It writes
$$2\lambda=(a^2+b^2+c^2+d^2)+\sqrt{(a^2+b^2+c^2+d^2)^2-4(bc-ad)^2}\leq2\times1.$$
Paradoxically, it not easy to deduce $|a|,|b|,|c|,|d|\leq 1$ from this condition.
\end{remark}

\subsection{Co-adapted couplings on $\He$}
We now describe co-adapted Heisenberg Brownian motions. As seen before, a Brownian motion $\B$ is entirely determined by its two first coordinates $(B_t)_t=(B_t^1,B^2_t)_t$; the third one being $(A_t)_t$ the L\'evy area swept by this 2-dimensional process $(B_t)_t$.

Thus two Heisenberg Brownian motions  $(\B_t)_t=(B_t^1, B_t^2, A_t)_t$ and  $(\B'_t)_t=(B'^1_t, B'^2_t, A'_t)_t$  on $\He$ are co-adapted if and only if $B =(B_t^1,B_t^2)_t$ and $B' =(B'^1_t,B'^2_t)_t$ are two co-adapted Brownian motions on $\R^2$ and if moreover their third coordinates satisfy
 \[
  dA_t=\frac{1}{2} \left( B^1_t dB^2_t - B^2_t dB^1_t\right) 
  \] and 
  \[
  dA'_t=\frac{1}{2} \left( B'^1_t dB'^2_t - B'^2_t dB'^1_t\right.).
 \]
 
For the following, we denote by $J$ and $\hat J$ the matrices appearing in Lemma \ref{lem:gen-coupling}.

A computation gives:
\[
 \B'^{-1}_t \B_t=\left( B_t^1-B'^1_t, B_t^2-B'^2_t, B_t^3-B'^3_t - \frac{1}{2} \left( B^1_t B'^2_t - B^2_t B'^1_t \right) \right)
\]

and thus:
 
 \[
  d(\B'^{-1}_t \B_t)= \left(\begin{array}{c}
                           dB_t^1-d B'^1_t\\
                           dB_t^2-d B'^2_t\\
                           ( B_t^1-B'^1_t) \left(\frac{dB'^2_t+ dB_t^2} {2} \right) - (B_t^2-B'^2_t)\left(\frac{dB'^1_t+ dB_t^1} {2}\right) - \frac{1}{2}(d\langle B^1_t, B'^2_t \rangle - d\langle B^2_t , B'^1_t\rangle  )
                          \end{array}\right),
 \]
 where we used:
\[
 d(X_t Y_t)=X_t dY_t + Y_t dX_t+ d\langle X_t,Y_t\rangle.
\]

We denote by $R_t$ the horizontal distance between the two Brownian motions $B_t$ and $B'_t$  in $\R^2$, that is $R_t^2=(B^1_t-B'^1_t)^2 + (B^2_t-B'^2_t)^2$  and by $Z_t$ the third coordinate, the relative L\'evy area. Hence $Z_t= (\B'^{-1}_t \B_t)_3$.

The homogeneous distance $d_H(\B_t,\B'_t)$ is thus given by 
\[
\sqrt{R_t^2 + |Z_t|.}
\]
In the following, when $R_t>0$, we choose to work in the direct orthonormal (random moving) frame $(v_1,v_2)$ defined by taking $v_1(t)$ the normalised vector of $\R^2$ directed by  $B_t-B'_t$. Let $ Q_t$ be the matrix whose columns are respectively $v_1(t)$ and $v_2(t)$. In this new basis, for $(\alpha,\beta)\in \R^2$ and  $( \cdot \mid \cdot )$ the usual scalar product on $\R^2$, we have:
  \begin{align*}\label{eq:co-adapted-K}
(\alpha v_1 + \beta v_2 \mid  dB'_t ) 
&= \left( Q_t \begin{pmatrix} \alpha \\ \beta \end{pmatrix} \mid J_t d B_t  + \hat J_t  d \hat B_t\right)\\
  &= \left(  \begin{pmatrix} \alpha \\ \beta \end{pmatrix} \mid  (Q_t^TJ_tQ_t) Q_t^T d B_t  + (Q_t^T\hat J_tQ_t) Q_t^T d \hat B_t\right)\\
  &= \left(  \begin{pmatrix} \alpha \\ \beta \end{pmatrix} \mid  K_t d W_t  + \hat K_t  d \hat W_t \right)
 \end{align*}
 for $K_t= Q_t^T J_t Q_t$ and $\hat K_t= Q_t^T \hat J_t Q_t$,  and where $W$ and $\hat W$ are the two standard independent 2-dimensional Brownian motions  defined by 
  \[
  dW_t = Q_t^T  dB_t , \;   d\hat W_t = Q_t^T  d\hat B_t. 
 \]
 This can be summed up as follows:
 \[
 Q_t^T dB'_t=\underbrace{(Q_t^TJ_tQ_t)}_{K_t} Q_t^TdB_t  + \underbrace{(Q_t^T\hat J_tQ_t)}_{\hat K_t} Q_t^Td\hat B_t.
 \]
  The next easy lemma describes the relation between the matrices $J$ and $K$.
 \begin{lemme}\label{lem:JK}
 With the above notation, when $R_t>0$,
 \begin{itemize}
 \item  Equation \eqref{eq:co-BM} is satisfied for $(K,\hat K)$ if and only if it is satisfied for $(J,\hat J)$.
 \item $\tr\,  K= \tr \, J$.
 \item  $ K^{1,2}-K^{2,1}=J^{1,2}-J^{2,1}$.
 \end{itemize}
 \end{lemme}

\begin{proof} 
The first two relations follow from the fact that $Q$ is an orthogonal matrix.
For the last relation, one  can  note that $J^{1,2}-J^{2,1}= \tr (  J M )=\tr (K Q^TM Q)$ with  $M$ the matrix $(\begin{smallmatrix}0&-1\\1&0\end{smallmatrix})$.
 Now a computation gives $Q^TM Q= (\det Q) M$ and the last relation follows from the fact $Q$ is actually a rotation matrix.
\end{proof}

\

 The stochastic processes $R_t^2$ and $Z_t$ are semimartingales defined for all time $t\geq 0$. In the next statement, we provide stochastic differential equations for their evolution. 
 
\begin{lemme}\label{lem:SDE}
With the above notation, when $R_t\neq 0$, the processes $R_t^2$ and $Z_t$ solve the stochastic differential  equation:
\begin{align*}
\left\{
\begin{aligned}
d(R^2_{t})&= 2R_t \sqrt{2(1-K^{1,1})}\, dC_t + \left( 2(1-K^{1,1}) + 2(1-K^{2,2}) \right)dt\\
dZ_{t} &=\frac{R_t}{2} \, \sqrt{2(1 +  K^{2,2})}  \, d\tilde C_t + \frac{1}{2}( K^{1,2} - K^{2,1} )dt\\
\end{aligned}
\right.
\end{align*}
where   $(C_t)_{t\geq 0}$ and  $(\tilde C_t)_{t\geq 0}$ are some  1-dimensional Brownian motions whose
co-variation  satisfies:
\begin{equation}\label{SDE:C-C}
\langle\sqrt{2(1-K^{1,1})}dC_t,\sqrt{2(1+K^{2,2})}d\tilde{C}_t\rangle=(K^{1,2}-K^{2,1})dt .
\end{equation}
\end{lemme}

\begin{remark}\label{rem:vraicoord}
Actually the stochastic process $(R_t^2,Z_t)_{t\geq 0}$ is perfectly defined for all $t\geq 0$ (even when $R_t=0$). The technical problem in Lemma \ref{lem:SDE} is that the matrix $Q_t$ and thus the matrix  $K_t$ are only defined for $R_t\neq  0$. However, the matrix $J_t$ is defined for every value of $R_t$ and we have:

\begin{align*}
\left\{
\begin{aligned}
d(R^2_{t})&= \sigma_R( B_t, B'_t, J_t) \, d C_t + \left(  2(1-J^{1,1}) + 2(1-J^{2,2}) \right) dt\\
dZ_{t} &=  \sigma_Z( B_t, B'_t, J_t) \, d \tilde C_t  +  \frac{1}{2}( J^{1,2} - J^{2,1} )dt
\end{aligned}
\right.
\end{align*}
where $\sigma_R$ and $\sigma_Z$ are defined by:
\[
 \sigma_R( B_t, B'_t, J_t)= \left\{ \begin{array} {lcc}
 0 & \textrm{ if }  &  B_t= B'_t\\
2 R_t \sqrt{2( 1- (Q_t^T J_t Q_t)^{1,1})} & \textrm{ if }  &  B_t \neq B'_t\\
 \end{array}\right.
 \]
 and \[
 \sigma_Z( B_t, B'_t, J_t)= \left\{ \begin{array} {lcc}
 0 & \textrm{ if }  &  B_t= B'_t\\
\frac{R_t}{2} \sqrt{2( 1+(Q_t^T J_t Q_t)^{2,2})} & \textrm{ if }  &  B_t \neq B'_t.\\
 \end{array}\right.
 \]
Note finally that the fact that $\sigma_R$ and $\sigma_Z$ vanish for $R_t=0$ is rather clear from their expressions in Lemma \ref{lem:SDE}.
\end{remark}

\begin{proof}[Proof of Lemma \ref{lem:SDE}]
The computations are done in \cite{K}  but we repeat them for the sake of completeness.

 First by It\^o formula and with the previous notation:
 \begin{align*}
  dR^2_t&=d\left( (B^1_t -B'^1_t )^2 + (B^2_t -B'^2_t )^2 \right)\\
                                                           &= 2 R_t  \left(v_1 \mid (dB_t-dB'_t) \right) + d\langle (B^1_t -B'^1_t ) , (B^1_t -B'^1_t ) \rangle + d\langle (B^2_t -B'^2_t ) , (B^2_t -B'^2_t ) \rangle .
 \end{align*}
We turn to the  martingale part and  write
  \begin{align*}
 \left(v_1 \mid (dB_t-dB'_t) \right)&=  \left( (K^{1,1}-1) dW_t^{1}+ K^{1,2} dW_t^{2} + \hat K^{1,1} d \hat  W_t^{1}+  \hat K^{1,2} d \hat W_t^{2} \right)\\
                        & =  \sqrt{2(1-K^{1,1})} dC_t
 \end{align*}
 for some  1-dimensional Brownian motion $(C_t)_t$ where we used Lemma \ref{lem:JK} for
\[
 (K^{1,1})^2 +   (K^{1,2})^2+ ( \hat K^{1,1})^2+   ( \hat K^{1,2})^2=1.    
\]

The quadratic variation writes 
 \begin{align*}
d\langle (B^1_t -B'^1_t ) , (B^1_t -B'^1_t ) \rangle  \rangle &= (J^{1,1}-1)^2+ (J^{1,2})^2 + ( \hat J^{1,1})^2+  ( \hat J^{1,2})^2= 2- 2 J^{1,1}
 \end{align*} 
and similarly
 \begin{align*}
d\langle (B^2_t -B'^2_t ) , (B^2_t -B'^2_t ) \rangle  \rangle &= (J^{2,1})^2+ (J^{2,2}-1)^2 + (  \hat J^{2,1})^2+ ( \hat J^{2,2})^2= 2- 2 J^{2,2},
 \end{align*}
 thus 
\begin{align*}
 d\langle (B^1_t -B'^1_t ) , (B^1_t -B'^1_t ) \rangle + d\langle (B^2_t -B'^2_t ) , (B^2_t -B'^2_t ) \rangle  &= 2 \tr(I-J)
                                            = 2\tr(I-K). 
 \end{align*}

We turn now to $Z_t$. Using the basis $(v_1,v_2)$, we can rewrite
  \begin{align*}
   dZ_t =\frac{R_t}{2}  \left(v_2 \mid  (dB'_t+ dB_t) \right) - \frac{1}{2}(d\langle B^1_t, B'^2_t \rangle - d \langle B^2_t , B'^1_t\rangle  ).
  \end{align*}
As before, we  get:
  \begin{align*}
  \left(v_2 \mid  (dB'_t+ dB_t) \right)  & = K^{2,1} dW^{1}_t+(K^{2,2}+1) dW^{2}_t +  \hat K^{2,1} d\hat W^{1}_t+  \hat{K}^{2,2} d\hat W^{2}_t\\
                                 & = \sqrt{2(1+K^{2,2})}  d\tilde C_t
  \end{align*}
 for some  1-dimensional Brownian motion $(\tilde C_t)_t$. 
Moreover:
\[
 d\langle B^1_t, B'^2_t \rangle - d\langle B^2_t , B'^1_t\rangle = (J^{2,1} -J^{1,2}  )dt = (K^{2,1} -K^{1,2}  )dt .
\]

The equation on the covariation \eqref{SDE:C-C} follows since by \eqref{eq:co-BM} and Lemma \ref{lem:JK},
\[
 K^{1,1}  K^{2,1}+  K^{1,2}  K^{2,2} + \hat K^{1,1}  \hat K^{2,1}+   \hat K^{1,2}  \hat K^{2,2}=0. 
\]
\end{proof}

\begin{remark}\label{rem:autreseq}
Since we will use them in the following we also write stochastic differential equations satisfied by $R_t$, $R_t^4$ and $Z_t^2$, obtained
for $R_t\neq 0$ using It\^o's formula in Lemma \ref{lem:SDE}:
\begin{align*}
\left\{
\begin{aligned}
dR_{t}&=\sqrt{2(1-K^{1,1})} \, dC_t +\frac{1-K^{2,2}}{R_t} dt,\\
d R_t^4&=4R_t^3 \sqrt{2(1-K^{1,1})}\, dC_t + (4 R_t^2 (1-K^{2,2}))dt + 12 R_t^2 (1-K^{1,1}) dt,\\
d(Z_t^2)&= Z_t R_t \sqrt{2(1 +  K^{2,2})}\,  d\tilde C_t + \left( Z_t( K^{1,2} - K^{2,1} ) + \frac{R_t^2}{2} (1+  K^{2,2}) \right) dt.
\end{aligned}
\right.
\end{align*}
As in Remark \ref{rem:vraicoord}, an expression for $dZ^2_t$ is possible in the canonical basis with the matrix $J$ in place of $K$. According to Lemma \ref{lem:JK}, $K^{1,2}-K^{2,1}$ is replaced by $J^{1,2}-J^{2,1}$ and the factor $R_t$ make the undefined terms vanish when $R_t=0$. The same holds for the semimartingale $(R^4_t)_{t\geq 0}$. On the contrary $R_t$ is not a semimartingale, but only locally when $R_t>0$.
\end{remark}

\subsection{Description of some couplings}
In this section, we describe some interesting couplings.

\subsubsection{The synchronous coupling ($K^{1,1}=1$, $K^{2,2}=1$, $K^{1,2}=K^{2,1}=0$.)}\label{sub:synchro}
The coupling is called synchronous because the planar trajectories $(B_t)_{t\geq 0}$ and $(B'_t)_{t\geq 0}$ are parallel. Here $R_t \equiv R_0 $ and $Z_t=Z_0+ W_{R_0 t}$ with $W$ a Brownian motion.

\subsubsection{The reflection coupling ($K^{1,1}=-1$ $K^{2,2}=1$, $K^{1,2}=K^{2,1}=0$.)}\label{reflec}
For the reflection coupling the planar trajectories of $(B_t)_{t\geq 0}$ and $(B'_t)_{t\geq 0}$ evolve symmetrically with respect to the bisector of line segment $[B_0,B'_0]$. We stop the coupling when $R_t$ hits $0$ and continue synchronously with $J^{1,1}=J^{2,2}=1$. Denote by $\tau$ this hitting time.
One thus has
\[
R_t=R_0 + 2C_{t\wedge \tau} \textrm{ and } Z_t= Z_0+  \int_0^{t\wedge \tau} (R_0 + 2C_{s\wedge\tau} )d \tilde C_s 
\]
where $(C_s)_s$ and   $(\tilde C_s)_s$ are two independent Brownian motions (starting in 0) and with $\tau= \inf\{ s\geq 0, 2C_s=-R_0\}$. 
This coupling is studied in Section 4. On Euclidean and Riemannian manifolds, the efficiency of reflection coupling has been studied in \cite{HS13,Ku07}.

\subsubsection{Kendall's coupling:  ($K^{1,1}= \pm 1$, $K^{1,2}=K^{2,1}=0$ and $K^{2,2}=1$).} In \cite{K}, Kendall describes a coupling which alternates between synchronous coupling and reflection coupling. In order to avoid the use of local times the strategy of Kendall is defined with hysteresis. The regime swaps when the process $(R_t,|Z_t|)$ hits a certain parabola $\{8Z_t^2= \kappa^2 R_t^4\}$ or $\{8Z_t^2=(\kappa-\eps)^2 R_t^4\}$ (see \cite[Theorem 4]{K}), depending for the synchronous or the reflection coupling. Thus the process is not Markovian, but it is co-adapted. The author proves that this coupling is  successful: this means $T:=\inf\{ s\geq 0, \B_s=\B'_s\}$ is almost surely finite, or, equivalently, the  process $(R_t, Z_t)$ hits almost surely $(0,0)$ in finite time.

\subsubsection{The perverse coupling: $K^{1,1}=1$, $K^{2,2}=-1$, $K^{1,2}=K^{2,1}=0$.}
We assume $R_0>0$. It satisfies,
\[
 dR(t)=\frac{2}{R_t} dt \textrm{ and } dZ_t=0.
\]
 
 Thus the distance $R_t$ and $Z_t$ are deterministic and given by:  
 \[
  R_t =\sqrt {R_0^2 + 4t} \textrm{ and } Z_t=Z_0.
 \]
The name perverse coupling is given by Kendall in \cite{Ke_ecp} as a generic name for a repulsive coupling. Here, the planar components of $(\B_t)_t$ and $(\B'_t)_t$ are coupled in a perverse way. This particular method to produce a perverse coupling appears in \cite[Section 5]{PaPo} in a Riemannian setting.

\section{Non-existence of co-adapted Heisenberg Brownian motions at bounded distance.}\label{sec:sec}
We now turn to the proofs of  Theorems \ref{thm:Winfty} and \ref{thm:W2}. 
\subsection{Proof of Theorem \ref{thm:Winfty}}
 Theorem \ref{thm:Winfty} is clearly  a corollary of Theorem \ref{thm:W2} but we can prove it more easily and it already shows a clear difference with the Riemannian case. Hence, we first present a proof of this result.

\begin{proof}[Proof of Theorem \ref{thm:Winfty}.]
 Assume that $(\B_t)_t$ and  $(\B'_t)_t$  are two co-adapted Heisenberg Brownian motions starting in $a=(x,y,z)$ and $a'=(x',y',z')$ with $R_0=\sqrt{(x'-x)^2+(y'-y)^2}>0$. Striving for a contradiction we suppose that $t\mapsto d_\He(\B_t,\B'_t)$ is almost surely and uniformly bounded. More precisely for some $C>0$ we assume $|R_t|+|Z_t|^{1/2}\leq C$ for every $t\geq 0$.
 
 Using Lemmas \ref{lem:JK} and \ref{lem:SDE} (or simply Remark \ref{rem:vraicoord}), we have 
\begin{align*}
  \E[R_t ^2] = R_0^2 +\E\left[ 2\int_0^T (1-J^{1,1}(s)) + (1-J^{2,2}(s)) ds \right]
 \end{align*}
and $R_t\leq C$ gives $\E[R_t ^2]\leq C^2$ and
\begin{align}\label{eq:Efini0}
  \E\left[ 2\int_0^{t} (1-J^{1,1}(s)) + (1-J^{2,2}(s)) ds \right]\leq C^2.
\end{align}
Recall from Lemma \ref{lem:JK} that $K^{1,1}+K^{2,2}=J^{1,1}+J^{2,2}$ and from Lemma \ref{lem:deuxdeux} that the matrix entries are $\geq -1$. Therefore 
 \begin{align}\label{eq:Efini1}
  \max_{i\in\{1,2\}}\E\left[ \int_0^{t} \left((1-K^{i,i}(s))\right) \mathbf 1_{ \{R_s>0\} }  ds \right]\leq C^2/2
\end{align}
and $R_t\leq C$, again, gives
 \begin{equation}\label{eq:Efini2}
  \E\left[ \int_0^{t}   (1-K^{2,2}(s)) \frac{R_s^2}{2}  ds \right]\leq C^4/4.
 \end{equation}
Until now we have used $\E(R_t^2)\leq C^2$ and $R_t\leq C$. We turn to exploit $\E(Z_t^2)\leq C^4$.
Lemma \ref{lem:JK} and Remark \ref{rem:autreseq} give
\begin{align}\label{eq:Z2}
  \E[Z_t ^2] = Z_0^2 +\E\left[ 2\int_0^t Z_s( J^{1,2}(s) - J^{2,1}(s) )  ds  + \int_0^t \frac{R_s^2}{2} (1+  K^{2,2}(s))   ds \right]
 \end{align}
Adding \eqref{eq:Efini2} and \eqref{eq:Z2}, and using $\E(Z_t^2)\leq C^4$, we obtain
 \begin{equation}\label{eq:Efini3}
 \E\left[ 2\int_0^{t} Z_t( J^{1,2}(s) - J^{2,1}(s) ) ds  + \int_0^{t}R_s^2  ds \right] \leq (1+1/4)C^4.
 \end{equation}

Next, we aim to compare $\E\left[ 2\int_0^{t} Z_s( J^{1,2}(s) - J^{2,1}(s) ) ds\right]$ with $\E\left[\int_0^{t}R_s^2  ds \right]$ that both appear in \eqref{eq:Efini3}. On the one hand, since $Z_t$ stays bounded $\E[Z_t^2]$ is also bounded. Hence by Cauchy-Schwarz inequality (for the product measure on $\Omega\times [0,t]$):
 \begin{align} \label{un} \notag
   \left| \E\left[ \int_0^{t} Z_s( J^{1,2}(s) - J^{2,1}(s) )ds \right] \right| &\leq \left(\ \int_0^{t}  \E [Z_t^2] ds\right)^{1/2}
                                                                              \left(\int_0^{t}  \E (J^{1,2}- J^{2,1})^2ds\right)^{1/2}\\ \notag
                                                                 &\leq \left(\ \int_0^{t}  \E [Z_t^2] ds\right)^{1/2}
                                                                              \left(2\int_0^{t}  \E ((J^{1,2})^2+ (J^{2,1})^2)ds\right)^{1/2}\\ \notag
                                                                              &\leq \sqrt{C^4 t}   \left(2\int_0^{t}  \E (J^{1,2})^2 + \E (J^{2,1})^2ds\right)^{1/2}\\ 
                                                                              &\leq \sqrt{C^4 t}\cdot\sqrt{2C^2}=C^3\sqrt{2t}.
 \end{align}

The last estimate follows from Lemma \ref{lem:deuxdeux} (the rows and columns of $J$ have $L^2$-norm smaller than $1$), $1-J_{i,i}^2\leq (1-J_{i,i})(1+J_{i,i})$ and \eqref{eq:Efini0}: 
\begin{align*}
 \int_0^{t}  \E (J^{1,2})^2+\E (J^{2,1})^2 ds  &\leq  \int_0^{t}  \E [(1- (J^{1,1})^2)+(1- (J^{2,2})^2)] ds\\
                             &\leq  \int_0^{t}  2\E [(1-J^{1,1})+(1-J^{2,2})] ds \leq C^2.
\end{align*}
On the other hand, since $(R_t^2)_{t\geq 0}$ is a submartingale, $\E[R_s^2] \geq R_0^2$ and
\begin{align} \label{deux}
\int_0^{t}\E[R_s^2]ds \geq R_0^2 t.
\end{align}
Since $R_0>0$, \eqref{un} and \eqref{deux} provide a contradiction in \eqref{eq:Efini3} as $t$ goes to infinity.
\end{proof}

\begin{remark}\label{rem:W4} The proof of Theorem \ref{thm:Winfty} can be improved to show that any co-adapted Heisenberg Brownian motions can not stay bounded in $L^{4}$. In this proof, the only place where we fully use the fact that $R_t$ is uniformly bounded almost surely is \eqref{eq:Efini2}. At the other places we merely need that $\E[R_t^2]$ and $\E[Z_t^2]$ are bounded. But $\E[R_t^4]\leq C^4$ for every $t\geq 0$ is a sufficient assumption for \eqref{eq:Efini2} and, hence, for the proof. 

Indeed, by Lemma \ref{rem:autreseq} one has
 \[
  \E[R_t^4] = R_0^4 + \E\left[ \int_0^t 12 R_s^2(1-K^{1,1}(s)) + 4 R_s^2 (1-K^{2,2}(s)) ds \right].
 \]
This quantity is uniformly bounded by $C^4$ for every $t\geq 0$ so that \eqref{eq:Efini2} holds. (the bound in \eqref{eq:Efini2} can even be divided by two: $C^4/8$ in place of $C^4/4$).
\end{remark}

\subsection{Proof of Theorem \ref{thm:W2}}
To go beyond Theorem \ref{thm:Winfty}, we conduct a precise study of the expected total variation (or length in $L^1$) of the martingale part and of the drift part of $(Z_t)_{t\geq 0}$, the relative L\'evy area.
As before, the proof will be by contradiction. The principle is the following.
 We derive an upper bound for the drift part of $Z_t$ similar to  \eqref{un} from the proof of Theorem \ref{thm:Winfty}; and using   Lemma \ref{lem:MG} below, we provide a lower bound for the martingale part of $Z_t$.
 
 \begin{lemme}\label{lem:MG}
Let $(N_t)_{t\geq 0}$ be a continuous martingale  with $N_0=0$ and $p$ be in $]0,1[$. Then there exists $a_p>0$ such that for every positive real numbers $\beta$ and $h$, the estimate
 \[
  \P\left(\langle N  \rangle_h \geq \beta \right) \geq p
 \]
implies
\[
 \E[|N_h|] \geq a_p \sqrt {\beta}.
\]
\end{lemme}
 
 The proof of  Lemma \ref{lem:MG} is postponed at the end of the section.

\begin{proof}[Proof of Theorem \ref{thm:W2}.]

Let $C$ be $\sup_{t\geq 0}\sqrt{\max(\E(R_t^2),\E(|Z_t|))}$ and as before, assume $C<+\infty$ by contradiction. First recall
\begin{align*}
  \E\left(R^2_t\right) =  R_0^2 +\E\left(\int_0^{t}2[(1-J^{1,1})+(1-J^{2,2})]ds\right)\geq 0,
\end{align*}
which gives
\begin{align}\label{J11}
 \E\left(\int_0^{+\infty}[(1-J^{1,1})+(1-J^{2,2})]ds\right) \leq \frac{C^2}{2} <+\infty.
\end{align}
Let $T:=\inf\{t\geq 0, R_t=\frac{R_0}{2}\}$ be the hitting time  of $\frac{R_0}{2}$.
We show  that we can assume $\P(T=+\infty)>0$.
Suppose for the rest of this paragraph $\P(T<+\infty)=1$ and let  $S$ the finite random variable   defined by 
\[S=\int_0^T    2(1-K^{1,1})ds.\] 
Because of  the non-negative  drift in the stochastic differential equation of $R_{t}$, using  the Dambins-Dubins-Schwarz theorem (see e.g.  \cite{MR1725357}), the random variable  $S$ is greater in stochastic order than the hitting time of $\frac{R_0}{2}$ for a Brownian motion starting in $R_0$. This hitting time is almost surely finite but nonintegrable.

Thus $\E(S)=+\infty$ which contradicts \eqref{J11} (Recall from Lemma \ref{lem:JK} that $J^{1,1}+J^{2,2}=K^{1,1}+K^{2,2}$ and from Lemma \ref{lem:deuxdeux} that these quantities are $\geq -1$).

 Now, let us decompose the semimartingale $(Z_{t})_t=M_t-A_t$ into its martingale $M_t$ and its  bounded variation part $-A_t$. From Lemma \ref{lem:SDE}, we recall:
 
 \begin{align}\label{At}
  -A_t= \frac{1}{2} \int_0^t (J^{1,2}-J^{2,1} ) ds.
 \end{align}

Applying Cauchy--Schwarz inequality and following the same track as for \eqref{un} we obtain
\begin{align}\label{totvar}
\E\int_0^t | J^{1,2}- J^{2,1}|ds\leq{\sqrt{t}}\sqrt{\E\int_0^t (J^{1,2}- J^{2,1})^2ds}\leq {\sqrt{t}}\sqrt{2C^2}.
\end{align}
Remark now that the quantity on the left hand side is two times the expected total variation of $A_t$ on $[0,t]$.
 
We postpone the proof of the following result until the end of the (present) proof. It occurs as an application of Lemma \ref{lem:MG}: there exists $h>0$ such that for every $t>0$,
\begin{align}\label{martingale_part}
\E(|M_{t+h}-M_t|)\geq 10 C^2.
\end{align}
  
Since, we have assumed $E(|M_t-A_t|)\leq C^2$ for every $t\geq 0$, the triangle inequality implies $\E(|A_{t+h}-A_t|)\geq 8C^2$ for every $t\geq 0$.
The control of the expected total variation of $(A_t)$ expressed in \eqref{totvar}  and the lower estimate just proved give
\begin{equation}\label{eq:contradiction}
8C^2 n\leq \sum_{k=1}^n \E(|A_{kh}-A_{(k-1)h}|)\leq 
\frac{1}{2}  \E\int_0^{nh} |J^{1,2}-J^{2,1} | ds  \leq \sqrt{\frac{C^2h n}{2}},\end{equation}

which, as $n$ tends to $\infty$, provides a contradiction with our initial assumption that was $\sup_{t\geq 0}\sqrt{\max(\E(R_t^2),\E(|Z_t|))}\leq C$. We are left with the proof of \eqref{martingale_part} (under the assumption of the $L^2$ boundedness).

Recall $T=\inf\{t\geq 0:\, R_t=\frac{R_0}{2}\}$ and set $q=\P(T=+\infty)$. We have already proved $q>0$.
We shall show further that for $h\geq \frac{5C^2}{q}$,  
\begin{equation}\label{crochet-Mt}
 \P\left(\langle M\rangle_{t+ h} - \langle M\rangle_{t}  \geq \frac{R_0^2h}{8} \right)  \geq   q- \frac{C^2}{2h}  \geq \frac{9q}{10}.
\end{equation} 
We hence obtain \eqref{martingale_part} taking $h$ large enough in \eqref{crochet-Mt} and applying Lemma \ref{lem:MG} to $(M_{t+h}-M_t)_{h\geq 0}$.

Proof of \eqref{crochet-Mt}: considering only the event $\{T=+\infty\}$ for the martingale part of $Z_t$ described in Lemma \ref{lem:SDE}, one has
\begin{align}\label{eq:yoh}
 \langle M\rangle_{t+ h} - \langle M\rangle_{t} \geq \mathbf 1_{\{ T=+\infty\}}\left(\frac{R_0}{2}\right)^2  \int _t^{t+h}  \frac{1+K^{2,2}}{2} ds.  
\end{align}

Now, since by \eqref{J11} it holds
\[
 \E\left[\int_t^{t+h}  \mathbf 1_{\{ T=+\infty\}} \frac{1-K^{2,2}}{2}\, \mathrm{d} s\right]\leq C^2/4,
\]
taking the complementary set of $\{\int _t^{t+h}  \frac{1+K^{2,2}}{2} ds  \geq  \frac{h}{2}\}$ in $\{T=+\infty\}$ and using Markov inequality, one obtains:
\begin{align}\label{eq:yohyoh} \notag
  &\P\left(   \mathbf 1_{\{ T=+\infty\}} \int _t^{t+h}  \frac{1+K^{2,2}}{2} ds  \geq  \frac{h}{2}\right)\\ 
 = & q- \P\left( \mathbf 1_{\{ T=+\infty\}}  \int_t^{t+h}  \frac{1-K^{2,2}}{2} ds  > \frac{h}{2}\right) \geq  q- \frac{C^2}{4}\cdot\frac2h.
\end{align}

Hence, in \eqref{eq:yoh} we consider the probability that the right-hand side is greater than $(R_0/2)^2\cdot(h/2)$, which, with \eqref{eq:yohyoh}, gives the wanted estimate \eqref{crochet-Mt} for every $h\geq \frac{5C^2}{q}$.
\end{proof}

\begin{proof}[Proof of Lemma \ref{lem:MG}.]
Let $\phi$ be the quadratic variation of $N$
$$\phi(t) =\langle N  \rangle_t,$$
and consider the hitting time $\tau=\inf\{t\geq 0:\, \phi(t)\geq \beta\}$. Set $\psi(t)= \phi(t) \wedge \beta$.
The Dambins theorem shows that there exists a standard Brownian motion $(W_t)_{t\geq 0}$ such that for every $t\geq 0$,
\begin{align}\label{N_to_B}
 \E\left[|N_t|\right] \geq \E\left[ |N_{t\wedge \tau}| \right]  = \E\left[ | W_{\psi(t)}| \right].
\end{align}

Let now $A$ be the event $\{\omega\in \Omega:\, \phi(h)\geq \beta \}$ and recall the assumption $\P(A)\geq p$. One has
\begin{align}\label{event_A}
  \E\left[ | W_{\psi(h)}| \right] \geq   \E\left[ | W_{\psi(h)}|\cdot\mathbf1_{A} \right] =   \E\left[ | W_{\beta}|\cdot  \mathbf1_{A} \right] \geq a_p\, \sqrt{ \beta}
  \end{align}
where the constant $a_p$  is given by 
$$a_p=\E\left[ |G| \, \mathbf1_{\{|G| \leq \Phi^{-1}(\frac{1+p}{2})\}} \right]=\int_{\Phi^{-1}(\frac{1-p}2)}^{\Phi^{-1}(\frac{1+p}2)}|x|\frac{e^{-x^2/2}}{\sqrt{2\pi}}\mathrm{d} x.$$
with $G$ a standard normal random variable and $\Phi$ its cumulative distribution function. The lower bound in \eqref{event_A} is obtained for $\P(A)=p$ and the normal random variable $W_{\beta}$ of variance $\beta$ concentrated as much as possible close to zero on event $A$. Equation \eqref{N_to_B} for $t=h$ and \eqref{event_A} finally provide the wanted estimate.
\end{proof}

\begin{remark}\label{rem:concave}
The major constraint for generalising Theorem \ref{thm:W2} and its proof to a $L^p$-Wasserstein control for $p< 2$ is that $\E(|Z_t|)$ is replaced by $\E(|Z_t|^{p/2})$. Here $x\mapsto |x|^{p/2}$ is not convex when $p<2$ and Jensen's inequality does not apply.
\end{remark}

\subsection{Proof of Theorems \ref{thm:rapport-Winfty} and \ref{thm:rapport-W2}}
As said before, Theorems \ref{thm:rapport-Winfty} and \ref{thm:rapport-W2} are deduced from Theorems \ref{thm:Winfty} and \ref{thm:W2} and the use of the homogeneous dilations.   

 For a fixed time $T>0$ and $p\in (0,\infty]$,  we introduce $C_{T,p}$ to be the constant:
 \[
C_{T,p}:= \sup_{a\neq a'\in\He}   \inf_{\mathcal {A}_T^{(a,a')}}    \sup_{0 \leq t \leq T}  \frac{ \E[ d_{\He}^p(\B_t^a, \B_t^{a'}) ]^{\frac{1}{p}}  } { d_{\He} (a,a')}\in [0,+\infty];
\]
where, as before in Remarks \ref{rem:rapport-Winfty} and \ref{rem:rapport-W2}, $ \mathcal {A}_T^{(a,a')}$ denotes the set of  co-adpated couplings of $(\B_t^a)_{0\leq t\leq T}$ and  $(\B_t^{a'})_{0\leq t\leq T}$, starting respectively in $a$ and $a'$. When $p=+\infty$, the numerator is  $\mathrm{essup}_{(\Omega,\P)}d_\He(\B_t^a,\B_t^{a'})$. As noticed in these remarks we aim at proving $C_{T,p}=+\infty$ for $p\geq 2$.

The first key point is to show that, using dilations, this constant does not depend on $T$. For this, let $S>0$ be another fixed time.
The point is that if $(\B_t^a, \B_t^{a'} )_{0\leq t\leq T}$ is a co-adapted coupling of two Heisenberg Brownian motions on $[0,T]$ starting respectively in $a$ and $a'$;
 setting for $b\in\{a,a'\}$:
\[
\tilde \B_s ^{\tilde b} : = \dil_{\sqrt\frac{S}{T} }\left( \B_{\frac{sT}{S}}^{ b}\right), 
 \textrm{ for } 0\leq s \leq S; 
\]
then, 
$
(\tilde \B_s^{\tilde a} , \tilde \B_s ^{\tilde{a'}})_{0\leq s\leq S}
$
is a co-adapted  coupling of two Heisenberg  Brownian motions on $[0,S]$ starting respectively in $\tilde a=\dil_{\sqrt \frac{S}{T}} (a) $ and $\tilde{a'}= \dil_{\sqrt \frac{S}{T}} (a')$,
Moreover,
\[
d_\He(\tilde \B_s ^{\tilde a}, \tilde \B_s ^{\tilde {a'}})
= \sqrt\frac{S}{T} d_\He \left( \B_{\frac{sT}{S}}^{a},  \B_{\frac{sT}{S}}^{a'}\right) \]
and 
\[ d_\He( \tilde a, \tilde {a'})= \sqrt \frac{S}{T}  d_\He (a,a').\]
This easily gives $C_{S,p} \leq C_{T,p}$ and by symmetry of $S$ and $T$: $ C_{S,p} = C_{T,p}$.

We can now turn to the proof of  Theorems \ref{thm:rapport-Winfty} and \ref{thm:rapport-W2}. Since the proofs are similar and the case $p=+\infty$ is easier, we only treat the case $p=2$.  
\begin{proof}[Proof of Theorem \ref{thm:rapport-W2}.]
Suppose by contradiction that $C_{T_0,2} <+\infty$ for some $T_0>0$. The above discussion implies that   $C_{T,2} <+\infty$ for each fixed time $T>0$.
Let $a=(0,0,0)$ and $a'=(x',0,0)$ with $x'\neq 0 $. Thus,  there exists a constant  $C$ (one can take $C= C_{T_0,2} \, d_\He(a,a')$) such that for each time  $T_0>0$,   there is a co-adapted coupling  satisfying $ \E[d_\He^2( B_t ^a, B_t^{a'})] \leq C$ for $t\in [0, T_0]$.

Now with the same notation as in  the proof of Theorem \ref{thm:W2} and denoting $q_{T_0}= \P(\forall 0\leq s\leq T_0,  R_s\geq \frac{R_0}{2} )$; one has $q_{T_0} \geq q$ 
and as before, there exists $h$ (independent of $T_0$) such that for all $0\leq t\leq T_0 -h $, 
\[
\E[|M_{t+h}-M_t|] \geq 10 C^2.
\]
Of course this gives: $\E[|A_{t+h}-A_t|]\geq 8C^2$ for every $0\leq t\leq T_0 -h $. Therefore equation \eqref{eq:contradiction} still holds if $nh\leq T_0$. Since the constants $C$ and $h$ are independent of $T_0$, letting $T_0$  and $n$ tend to infinity gives the contradiction. Thus $C_{T_0,2} =+\infty$.
\end{proof}

\section{Coupling by reflection}\label{sec:refl}
In this section, we study precisely the coupling by reflection.
We recall that $(\B_t)_{t\geq 0}$ and $(\B'_t)_{t\geq 0}$ are two Heisenberg Brownian motions coupled by  reflection if and only if
their horizontal parts $(B_t)_{t\geq 0}$ and $(B'_t)_{t\geq 0}$ are two Brownian motions on $\R^2$ coupled by reflection. This means that the coupling matrices are  given by 
$K^{1,1}=-1$, $K^{2,2}=1$, $K^{1,2}=K^{2,1}=0$ for $t< \tau$ and by the matrix $J=\mathrm{Id}_2$ for $t\geq\tau$ where $\tau= \inf \{s\geq 0:\, R_s =0\}$ is the hitting time of $0$ for $(R_t)_{t\geq 0}$.
We recall 
\[
R_t=R_0 + 2 C_{t\wedge \tau}\quad \textrm{ and }\quad Z_t= Z_0+  \int_0^{t\wedge \tau} (R_0 + 2 C_{s\wedge\tau} )d \tilde C_s 
\]
where $(C_s)_s$ and   $(\tilde C_s)_s$ are two independent Brownian motions (starting in 0) and with $\tau= \inf\{ s\geq 0:\, C_s=-R_0/2\}$. 

For simplicity, in the following we only consider the case $R_0>0$ and $Z_0=0$.
\begin{prop}\label{pro:assymp}
With the above notation, assume $R_0>0$ and $Z_0=0$,
 Let $p>0$, then  there exists some constants $C_p, C_p', C_p''>0$ such that 
\[
\E[R_t]= R_0
\]
and 
\[
\left\{ \begin{array}{l l l}
\E[|Z_t|^p]& \sim_{t\to \infty} C_p R_0 t^{p-\frac{1}{2}} & \textrm{ if } p > \frac{1}{2}\\
\E[|Z_t|^p] &\sim_{t\to \infty} C_p' R_0 \ln t &\textrm{ if } p =\frac{1}{2} \\
\E[|Z_t|^p] &\to_{t\to+\infty}  C_p'' R_0^{2p}  & \textrm{ if } 0 < p <\frac{1}{2}.
   \end{array} \right.
\]
\end{prop}

\begin{remark}
In particular for $0<\alpha<1$, the upper bound
\[
\sup_{t\geq 0} \E\left[d_\He(\B_t, \B'_t)^\alpha\right] <+\infty
\]
is satisfied by the coupling by reflection. This is obtained by recalling that $d_\He(\B_t, \B'_t)$ is equivalent to the homogeneous distance $\sqrt{R_t^2 + |Z_t|}$ using Proposition \ref{pro:assymp} for $p=\alpha/2$, $(a+b)^p\leq a^p+b^p$, and Jensen's inequality $[\E(R_t)^\alpha]\leq [\E(R_t)]^\alpha$. \end{remark}

\begin{proof}
We assume $R_0=1$. Let $t>0$  be fixed. By the Dambins-Dunford-Schwarz theorem, $Z$ is  a changed time Brownian motion:
\[
Z_t= W_{T(t)} \textrm{ with } T(t)=\int_0^t R_s^2 ds
\] 
with $W$ a Brownian motion independent of $(R_t)_{t\geq 0}$. 
Set $\tau =\inf\{s\geq 0:\, R_s=0\}$. As $( R_s/2)_{s\geq 0}$ is a Brownian motion starting in $R_0/2$ and stopped in 0, it is known that $\tau$ is almost surely finite and that its  density  $f_\tau$ is given  by 

\begin{align}\label{densite}
f_\tau(u)=\frac{R_0/2}{ \sqrt{2\pi} u^{3/2}} e^{-\frac{R_0^2}{4u}}, \; u\geq 0.
\end{align}

Using $\tau$, we compute
\begin{align} \label{decomp}
E\left[|Z_t|^{p}\right]&=\E(|W_{T(t)}|^{p})\\ \notag
 &=\int_0^{+\infty} \E(|W_{T(t)}|^{p}|\tau=u)f_\tau(u)d u \\ \notag
&=\underbrace{\int_0^{t} \E(|W_{T(t)}|^{p}|\tau=u)f_\tau(u)d u}_{h_1(t)}+ \underbrace{\int_t^{+\infty} \E(|W_{T(t)}|^{p}|\tau=u)f_\tau(u)d u}_{h_2(t)}. \notag
\end{align}
In the last line, we split the integral between the trajectories of $R$ that have hit $0$ before $t$ and those which will hit 0 after $t$.

Let us estimate $h_1(t)$, the first integral in the decomposition \eqref{decomp}. Hence we set $u\leq t$. Since $W$ and $R$ are  independent, with $c_p=\E(|W_1|^{p})$, one has:
\begin{align}  \label{transfo}
 \E(|W_{T(t)}|^{p}|\tau=u)&= c_p  \E(|T(t)|^{p/2}|\tau=u)\\ \notag
                                            &= c_p \E\left( \left(\int_0^u R_s^2 ds\right) ^{p/2}  |\tau=u \right)\\ \notag
&=  c_p 2^p  u^p \E\left[ \left(\int_0^1 \tilde R_\lambda^2 d\lambda \right) ^{p/2}  |\tau=u \right]\notag
\end{align}
where we have introduced the normalised 
process   $(\tilde R_\lambda)_{\lambda\in [0,1]}$ (defined almost surely, since the hitting time $\tau$ is almost surely finite) in such a way it hits 0 at time 1:
\[
\tilde R_\lambda= \frac{1}{2 \sqrt \tau} R_{\tau \lambda},\quad \lambda \in [0,1].
\]
Note that $\tilde R_0= \frac{R_0}{2\sqrt \tau}$.

It is then well-known that, conditioned on $\tau=u$, the entire process $(\tilde R_\lambda)_{\lambda \in [0,1]} 
$ converges in law  when $u\to \infty$ to a normal positive Brownian excursion  $(X_s)_{s\in[0,1]}$.
Moreover, as proven in Lemma \ref{lem:unif-int}, when $u\to \infty$,

\begin{align}\label{stab}
 \E\left[ \left(\int_0^1 \tilde R_\lambda^2 d\lambda \right) ^{p/2}  |\tau=u \right] 
\to \E\left[ \left(\int_0^1 X_s^2 ds\right) ^{p/2}  \right].
\end{align}


\

\

Finally with \eqref{transfo} denoting the limit in \eqref{stab} by $E_p$, the first integral in \eqref{decomp} satisfies the following equivalence:
\begin{align*}
h_1(t)=\int_0^{t} \E(|W_{T(t)}|^{p}|\tau=u)f_\tau(u)d u & \sim_{t\to \infty}2^p  \, c_p \, E_p \int_0^{t} u^p f_\tau(u) du
\end{align*}
From the density estimate of $f_\tau$ in \eqref{densite} we have $u^pf_\tau(u)\sim_{+\infty} \frac{R_0}{2 \sqrt{2\pi}}u^{p-3/2}$. Therefore:
\begin{itemize}
\item If $p>1/2$, the function $h_1(t)$ is equivalent to $\frac{ 2^{p-3/2}\,  R_0 \, c_p E_p}{ (p-1/2) \sqrt {\pi} \,  u^{3/2}}t^{p-1/2}$ at $+\infty$,
\item if $p=1/2$, it is equivalent to $\frac{R_0c_p E_p}{2 \sqrt {\pi} \,  u^{3/2}}\ln t$,
\item if $0<p<1/2$, it converges to a positive constant.
\end{itemize}

\

We now turn to $h_2$. 
As before,
 \begin{align*}
 h_2(t)&=\int_t^{+\infty} \E(|W_{T(t)}|^{p}|\tau=u)f_\tau(u)d u\\
 &=  c_p \int_t^{+\infty}\E\left( \left(\int_0^t R_s^2 ds\right) ^{p/2}  |\tau=u \right)f_\tau(u)d u\\
&=2^p \,   c_p \,     \int_t^{+\infty}  u^p \E\left[ \left(\int_0^{ \frac{t}{u}} \tilde R_\lambda^2 d\lambda \right) ^{p/2}  |\tau=u \right] f_\tau(u)d u\\
&=2^p\,  c_p \,   t^{p+1} \,  \int_1^{+\infty}  v^p \E\left[ \left(\int_0^{ \frac{1}{v}} \tilde R_\lambda^2 d\lambda \right) ^{p/2}  |\tau= tv \right] f_\tau(tv)d v\\
&=   \frac{ 2^{p-3/2}\, c_p \, R_0 \, t^{p-3/2}}{ \sqrt{\pi}} \int_1^{+\infty}  v^{p-3/2} \E\left[ \left(\int_0^{ \frac{1}{v}} \tilde R_\lambda^2 d\lambda \right) ^{p/2}  |\tau= tv \right] e^{-\frac{R_0^2}{tv}} dv\\
\end{align*}
where, as above, $s=u\lambda$ and $\tilde{R}_{\lambda}=R_{\tau\lambda}/\sqrt{\tau}$ and where we set the change of variable $u=tv$ in the next to last line.

Now, Lemma \ref{lem:unif-int} and the dominated convergence, which is completely justified  by Lemma \ref{lem:bessel}, give as $t\to+\infty$,
\begin{align*}
  \int_1^{+\infty}  v^{p-3/2} \E\left[ \left(\int_0^{ \frac{1}{v}} \tilde R_\lambda^2 d\lambda \right) ^{p/2}  |\tau= tv \right] e^{-\frac{R_0^2}{4tv}} dv
\to   \int_1^{+\infty}  v^{p-3/2} \E\left[ \left(\int_0^{ \frac{1}{v}}X_s^2 ds \right) ^{p/2} \right] d v.
\end{align*}
As a consequence, denoting by $I_p$  the last integral,
\[
h_2(t) \sim_{t\to \infty} \frac{2^{p-3/2}\, c_p  \,  R_0}{\sqrt {\pi}}  t^{p-1/2} I_p.
\]
This with the treatment of $h_1$ above gives the complete result in case $R_0=1$. Next, if $R_0>0$ 
one infers  
$\E_{R_0} [|Z_t|^p ]=R_0^{2p}  \;   \E_{\{R_0=1\} }[| Z_{t/R_0^2} |^p]$
from the classical dilations of Subsection \ref{sec:H}, so that the general case follows.
\end{proof}

The two next lemmas complete the proof of Proposition \ref{pro:assymp}.
\begin{lemme}\label{lem:unif-int}
Let $( \tilde R_t)_{t\in[0,1]}$ be a Brownian motion starting in $r_0>0$ conditioned to hit 0 for the first time at time 1.
Let $\alpha\in [0,1]$ and $p$ be positive. As $r_0\to 0$, 

\begin{align}\label{stab2}
 \E\left[ \left(\int_0^\alpha \tilde R_\lambda^2 d\lambda \right) ^{p/2}\right] 
\to \E\left[ \left(\int_0^\alpha X_s^2 ds\right) ^{p/2}  \right]
\end{align}
where $( X_s)_{s\in [0,1]}$  is a Brownian excursion.
\end{lemme}
\begin{proof} The process $ (\tilde R_t)_{t\in[0,1]}$ converges in law to the Brownian excursion $( X_s)_{s\in [0,1]}$.
To obtain the convergence of the moments of $\int_0^\alpha \tilde R_\lambda^2 d\lambda$, we use a uniform integrability property. Let 0$ \leq \alpha\leq 1$. We bound 
\[
\P\left(  \left(\int_0^\alpha \tilde R_\lambda^2 d\lambda \right) \geq y |\,\tau=u\right) \leq \P\left( \sup_{0\leq s\leq 1} W_s \geq \sqrt y |\, T_0= 1\right)
\]
where $(W_s)_{s\geq 0}$ is a  Brownian motion starting in $r_0$ and $T_0$ its hitting time of $0$.
Next, by \cite[Formula 2.1.4 (1) p.198]{BoSa_second}, still for $W$ starting in $r_0$, we have for every $t>0$
\begin{align*}
\P\left( \sup_{0\leq s\leq T_0} W_s <y |\, T_0= t\right)&= \sum_{k=-\infty} ^{+\infty} \frac{( r_0 + 2ky)}{\sqrt {2\pi} t^{3/2}} \exp\left( -  \frac{( r_0 + 2ky)^2}{2t} \right)   \frac{\sqrt {2\pi} t^{3/2}}{r_0}  \exp\left(   \frac{ r_0^2}{2t} \right),
\end{align*}
In particular, reorganising the terms,
\begin{align*}
\P \left( \sup_{0\leq s\leq 1} W_s  <y |\, T_0= 1\right)&
&=1 - 2 \sum_{k=1} ^{+\infty} \left( 4k^2y^2 \frac{\sinh(2kyr_0)}{2kyr_0}-\cosh(2kyr_0)   \right) \exp\left(- 2k^2 y^2\right),
\end{align*}
and, since  for $u\geq 0$, $\frac{\sinh u}{u} \leq e^{u}$, uniformly on  $0<r_0\leq 1$ and $y\geq 1$, 
 \begin{align*}
\P \left( \sup_{0\leq s\leq 1} W_s \geq y |\,T_0= 1\right)
&\leq 8\sum_{k=1} ^{+\infty}  k^2y^2 \exp(2ky)\exp\left(- 2k^2 y^2\right) \leq a \exp(-by^2)
\end{align*}
for some $a,b>0$.

Thus, for all $0<r_0\leq 1$, the random variables  $ \left(\int_0^\alpha \tilde R_\lambda^2 d\lambda \right)_{r_0\leq 1}$ admit some uniformly bounded exponential moment. As a consequence, the corresponding $ \left(\int_0^\alpha \tilde R_\lambda^2 d\lambda \right)^\frac{p}{2}$ are uniformly integrable and the desired convergence follows. 
\end{proof}

\begin{lemme}\label{lem:bessel}
With the above notation,
there exists a coupling of $( \tilde R_t)_{t\in[0,1]}$  and of a 3-Bessel process $(V_t)_{t\in [0,1]} $ on the same probability space such that both start in $r_0\geq 0$ and such  almost surely:
\[
\tilde R_t \leq V_ t , \, \textrm{for all } 0\leq t\leq 1.
\]
In particular, for $R_0$ fixed and $p>0$ there exists a constant $D_p>0$ such that for all $v\geq 1$ and $r_0 \leq \frac{R_0}{\sqrt {v}}$ it holds 
\[
v^p \E\left[ \left(\int_0^{ \frac{1}{v}} \tilde R_\lambda^2 d\lambda \right) ^{p/2}\right] \leq D_p
\]
where the process  $\tilde R_\lambda$ starts in $r_0$.
\end{lemme}

\begin{proof}
The process $ \tilde R_t$ can be thought as a Bessel bridge (see e.g. \cite[Chapter XIII] {MR1725357}) and thus satisfies 
\[
d \tilde R_t = dW_t +  \left(\frac{1}{\tilde R_t} - \frac{\tilde R_t}{1-t} \right) dt
\]
for some Brownian motion $(W_t)_{t\in[0,1]}$.
The coupling is obtained by considering the same Brownian motion in the stochastic differential equation defining $(V_t)_{t \in [0,1]}$:
\[
d V_t = dW_t + \frac{1}{ V_t} dt
\]
The other conclusion follows since the 3-Bessel process shares the same scaling property as the Brownian motion: $(V_{\lambda t}/\sqrt{\lambda})_t$ has the same law as the 3-Bessel process starting in $r_0/\sqrt{\lambda}$. 
\end{proof}

\section{The static coupling, a transport problem}\label{sec:bonus}
In this section, we turn to the proof of Theorem \ref{thm:static}. Recall that it gives a direct proof of the following $L^1$-Wasserstein control:
There exists $C>0$ such that for every $t\geq 0$, and every $a,a'\in \He$,
\[\mathcal{W}_1(\mu^a_t,\mu^{a'}_t) \leq Cd_\He(a,a')\]
where $\mu^a_t=\law(\B^a_t)$ and $\mu^{a'}_t=\law(\B^{a'}_t)$.

\begin{proof}[Proof of Theorem \ref{thm:static} ]
\emph{Reduction of the problem:} Let us first see how, using the symmetries of $\He$ presented in Subsection \ref{sec:H}, the proof can be reduced to $t=1$, $a=(0,0,0)$ and $a'=(x',0,0)$. First, the isometries of $\He$ induced isometries for $\mathcal{W}_1$. In particular $(\trans_p)_\#$ and $(\rot_\theta)_\#$ are isometries for $\mathcal{W}_1$, for every $p\in \He$ and $\theta\in \R$, that stabilise the family $\{\mu^a_t\}_{a\in \He}$.
\[\mathcal{W}_1(\mu^a_t,\mu^{a'}_t)=\mathcal{W}_1((\trans_p)_\#\mu^a_t,(\trans_p)_\#\mu^{a'}_t)=\mathcal{W}_1(\mu^{p.a}_t,\mu^{p.a'}_t).\]
Hence, we can assume $a=(0,0,0)$. Using $\rot_\theta$ we can moreover assume $a'=(x',0,z')$. As moreover $\dil_\lambda$, defined in Subsection \ref{sec:H} satisfies $(\dil_\lambda)_\#\mu^a_t=\mu^{\dil_\lambda(a)}_{\lambda^2 t}$ we can assume $t=1$. Finally
\begin{align*}
\mathcal{W}_1(\mu^0_1,\mu^{(x',0,z')}_1)&\leq \mathcal{W}_1(\mu^0_1,\mu^{(x',0,0)}_1)+\mathcal{W}_1(\mu^{(x',0,0)}_1,\mu^{(x',0,z')}_1)\\
&\leq \mathcal{W}_1(\mu^0_1,\mu^{(x',0,0)}_1)+\mathcal{W}_1(\mu^{0}_1,\mu^{(0,0,z')}_1)\\
&\leq \mathcal{W}_1(\mu^0_1,\mu^{(x',0,0)}_1)+d_\He((0,0,0),(0,0,z')).
\end{align*}
The estimate $\mathcal{W}_1(\mu^{0}_1,\mu^{(0,0,z')}_1)\leq d_\He((0,0,0),(0,0,z'))=2\sqrt{\pi|z'|}$ comes from the fact that $\trans_{(0,0,z')}$ is not only the left-translation but also the right-translation of vector $(0,0,z')$. If $(X,Y,Z)\sim \mu^{0}_1$, then $(X,Y,Z).(0,0,z')\sim \mu^{(0,0,z')}_1$ so that
\[d_\He((X,Y,Z).(0,0,z'),(X,Y,Z))=d_\He((0,0,0),(0,0,z')).\]
If we can prove $\mathcal{W}_1(\mu^0_1,\mu^{(x',0,0)}_1)\leq C' d_\He((0,0,0),(x',0,0))=C'|x'|$ for some $C'\geq 1$, we have finally
\begin{align*}
\mathcal{W}_1(\mu^0_1,\mu^{(x',0,z')}_1)&\leq C d_\He((0,0,0),(x',0,0))+d_\He((0,0,0),(0,0,z'))\\
&=C' |x'|+2\sqrt{\pi |z'|} \leq\max( C', \sqrt{2\pi}) \sqrt{(x')^2+|z'|}\\
&\leq C d_\He((0,0,0),(x',0,z')),
\end{align*}
with $C= c\cdot \max( C', \sqrt{2\pi})$  with $c$ defined as in  \eqref{eq:dist-eq}. Finally the proof amounts to the case $t=1$, $a=(0,0,0)$ and $a'=(x',0,0)$, as we announced.

\emph{Main body of the proof:} We set 

\begin{align}
\mu&=\mu^0_1=\law(X,Y,Z)\label{ligne_un}\\
\nu&=\mu_1^{(x',0,0)}=\law\left[(x',0,0).(X,Y,Z)\right]=\law(X+x',Y,Z+(1/2)x'Y)\label{ligne_deux}\\
\tilde\mu&=\law\left[(x',0,0).(X,Y,Z).(-x',0,0)\right]=\law(X,Y,Z+x'Y) \label{ligne_trois}
\end{align}

We want to estimate $W_1(\mu,\nu)$  from above and start with
\[W_1(\mu,\nu)\leq W_1(\mu,\tilde \mu)+W_1(\tilde \mu,\nu)\]
The coupling in \eqref{ligne_deux} and \eqref{ligne_trois} yields $W_1(\tilde \mu,\nu)\leq |x'|=d_\He((0,0,0),(x',0,0))$.
The coupling suggested in \eqref{ligne_un} and \eqref{ligne_trois} yields 
\begin{align}
W_1(\mu,\tilde \mu)&\leq \E\left[\E\left[d_\He((X,Y,Z),(X,Y,Z+x'Y))|(X,Y)\right] \right]\label{restart}\\  \notag
&\leq \E\left[\E\left[d_\He((0,0,0),(0,0,x'Y))|(X,Y)\right]\right]\\
&\leq \E \left[2\sqrt{\pi |Yx'|}\right]=\sqrt{|x'|}\times[2\sqrt{\pi} \, \E(\sqrt{|Y|})] \notag
\end{align}
The order $\sqrt{|x'|}$ is such that $\sqrt{|x'|}/d_\He((x',0,0),(0,0,0))=|x'|^{-1/2}\to \infty$ as $x'$ goes to zero.

\begin{remark}
Note that the coupling in \eqref{restart} is exactly the synchronous coupling of subsection \ref{sub:synchro} evaluated at time $t=1$. The estimate from above can easily be checked with Lemma \ref{lem:SDE}.
\end{remark}

We modify the computation above just a little  based on the knowledge that the translation is not the optimal transport plan on the real line when considering costs that are increasing concave functions of the distance.

The following lemma will be in order:
\begin{lemme}\label{lem:sqrt}
If $\eta$ is a probability measure on $\R$ with rapidly decreasing and smooth density $f$, then
\[\inf_{\pi\in \Pi(\eta, (\trans^\R_t)_\#\eta)} \iint \sqrt{|y-x|}\, d \pi\leq |t| \times\left(\int |f'(x)|\sqrt{|x|}\,dx\right) \]
\end{lemme}
\begin{proof}
The left-hand side is the $1$-Wasserstein distance $\mathcal{W}^\R_1(\eta,(\trans^\R_t)_\#\eta)$ on $\R$ for the distance $(x,y)\mapsto \sqrt{|y-x|}$. This is also the Kantorovich norm $\|\eta-(\trans^\R_t)_\#\eta)\|_1$, in the sense of \cite{MR0102006}. Recall that the Kantorovich norm of a signed Radon measure $\sigma=\sigma_+-\sigma_-$ of mass zero is 
\[\|\sigma\|_1=\mathcal{W}_1^\R(\sigma_+,\sigma_-)=\inf_{\pi\in \Pi(\sigma_+,\sigma_-)}\iint \sqrt{|y-x|}d\pi(x,y)\]

where the mass of $\sigma_+$ can be different from 1. Let $\sigma=\eta-(\trans^\R_t)_\#\eta)$ and for the computation of $\|\sigma\|_1$ assume without loss of generality that $t\geq 0$. We have $\sigma=\int_0^t (\trans^\R_u)_\#\gamma\,d u$ where $\gamma$ is the Radon measure of density $f'$. Therefore 
\begin{align*}
\|\sigma\|_1=\left\|\int_0^t(\trans^\R_u)_\#\gamma\,du\right\|&\leq \int_0^t\|(\trans^\R_u)_\#\gamma\|_1\,du\\
&= \int_0^t\|\gamma\|_1\,du \,  = t \|\gamma\|_1 \\
&\leq \,  t\times\left(\int |f'(x)|\sqrt{|x|}\,dx\right).
\end{align*}
Note that for the last inequality, we use the triangle inequality by transporting $\gamma_+$, of density $f'_+$, to the atomic measure $(\int f'_+) \delta_0$ and then from this measure to $\gamma_-$, of density $f'_-$.  
\end{proof}
We can now conclude by coupling $Z$ and $Z'$, conditionally on $(X,Y)$. From \eqref{restart}:
\begin{align*}
\E\left[d_\He((X,Y,Z),(X,Y,Z+x'Y))|(X,Y)\right]&\leq 2\sqrt{\pi}|x'Y|\times\left(2\int |f'_{Z|X,Y}(z)|\sqrt{|z|}\,dz\right)\\
\end{align*}
where $f_{Z|x,y}$ is the density of $\law(Z|\,X=x,\ Y=y)$. Finally,
\begin{align*}
&\E\left[d_\He((X,Y,Z),(X,Y,Z+x'Y))\right]\\
&\leq |x'|\iint f_{X,Y}(x,y) 2\sqrt{\pi}|y|\times\left(\int |f_{Z|x,y}'(z)|\sqrt{|z|}\,dz\right)\,dx\,dy\\
&\leq |x'|\iiint 2\sqrt{\pi}|y|\sqrt{|z|}\left(\frac{|\partial_z f_{X,Y,Z}(x,y,z)|}{f_{X,Y,Z}(x,y,z)}\right)f_{X,Y,Z}(x,y,z)\,dx\,dy\,dz.
\end{align*}

Here $f_{X,Y,Z}$ is the density of $\mu=\law(X,Y,Z)$ with respect to the Lebesgue measure and $f_{X,Y}$ is the density of $(X,Y)$. Note that we have $f_{Z|x,y}(z)=f_{X,Y,Z}(x,y,z)/f_{X,Y}(x,y)$.

Since $f_{X,Y,Z}$ is the density of the heat kernel, it is well known (see \cite{Li} and \cite{BBBC} equation (14) ) that there exists a constant $C>0$ such that for all $(x,y,z)\in\He$,
\begin{equation}\label{eq:dzp1}
\left|  \frac{\partial_z f_{X,Y,Z}(x,y,z)}{f_{X,Y,Z}(x,y,z)} \right| \leq C.
\end{equation} 
The proof  of this fact is analytic and is based on the explicit representation of the heat kernel as  an oscillatory intergal.

The proof of Theorem \ref{thm:static} is then finished since the quantity $|y| \sqrt  {|z|}$
is clearly integrable with respect to the heat kernel $ f_{X,Y,Z}$.

\end{proof}

\section{Generalisation  to the Heisenberg groups of higher dimension.}\label{sec:gene}
In this section, we prove that Theorems \ref{thm:rapport-Winfty}, \ref{thm:Winfty},
 \ref{thm:rapport-W2}, \ref{thm:W2}, \ref{thm:refl} and \ref{thm:static} also hold in the case of the Heisenberg groups of higher dimension.
For $n\geq 1$ the Heisenberg group $\He_n$ can be identified with $\R^{2n+1}$ equipped with the law:
\[
  \left((x_i,y_i)_{i=1}^n ,z\right)\cdot  \left((x'_i,y'_i)_{i=1}^n ,z'\right)=   \left((x_i+x'_i,y_i+y_i')_{i=1}^n ,  z+z'+\frac{1}{2}\sum_{i=1}^n (x_iy_i'-y_ix_i')\right).
 \]
The corresponding Brownian motion starting from $0_{\R^{2n+1}}$ is given by:
\[
\B_t^0:= \left( \left(B_{t,i}^1, B_{t,i}^2\right)_{i=1,\dots,n},   \frac{1}{2} \sum_{i=1}^n\left( \int_0^t B^1_{s,i} dB^2_{s,i} -  \int_0^t B^2_{s,i} dB^1_{s,i}\right) \right)
\]
where  $B_t:=\left(B_{t,i}^1, B_{t,i}^2\right)_{i=1,\dots,n}$ is a $2n$-dimensional standard Brownian motion. As before, we denote by $(\B_t ^a)_{t\geq 0}$ the Brownian motion starting in $a\in \He_n$. It can be represented as  $a \cdot \B_t^0$ and we write $\mu_t^a$ for its law at time $t$.

Lemma \ref{lem:gen-coupling} can directly be generalised for describing co-adapted Heisenberg Brownian motions $(\B_t,\B_t')$ but with matrices $J, \hat J\in \mathcal M_{2n}(\R)$.
As above, we denote by $R_t$ the Euclidean norm of $B_t'-B_t\in \R^{2n}$ and by $Z_t $ the last coordinate of $\B_t'^{-1} \B_t$ that we still call the relative L\'evy area. The quantity 
\[
d_H(\B_t,\B'_t):= \sqrt{R_t^2 + |Z_t|}
\]
is still a homogenous distance on $\He_n$ and is equivalent to the Carnot-Carath\'eodory distance $d_\He$.

When $R_t>0$ we introduce the following basis:
let $e_1$ be $\frac{1}{R_t}(B_t' - B_t)\in \R^{2n}$, write $e_1=(a_1, \dots a_n), \, a_j\in \mathbb C=\R^2$,  set $e_2= (ia_1,\dots, i a_n)$ and complete $(e_1,e_2) $ into a direct orthonormal basis of $\R^{2n}$. 
This basis is well adapted for studying couplings in $\He_n$.
Indeed, with $L$ and $\hat L$ being the coupling matrices in this new basis in place of $J, \hat J$ in the canonical basis, a computation gives:

\begin{lemme} With the above notation, if $R_t>0$, then
\begin{align*}
\left\{
\begin{aligned}
d(R^2_{t})&=2  R_t \sqrt{2(1-L^{1,1}) } \, dC_t + 2  \, \tr \left(I_{2n}-   J\right)dt\\
dZ_{t} &= \frac{1}{2}\, R_t \sqrt{2(1+L^{2,2}) } \,  d\tilde C_t + \frac{1}{2} \sum_{i=1}^n ( J^{2i-1,2i} - J^{2i,2i-1} )dt\\
\end{aligned}
\right.
\end{align*}
where   $(C_t)_{t\geq 0}$ and  $(\tilde C_t)_{t\geq 0}$ are some 1-dimensional (possibly correlated) standard Brownian motions.  
\end{lemme}

Now, since $\tr L= \tr J$, and since each $|L^{i,i}|  \leq 1$ for $1\leq i\leq n$,
 \[
 1-L^{2,2} \leq \tr(I-L)= \tr(I-J)
 \]
  and one can directly adapt the proof of  Theorems \ref{thm:Winfty} and  \ref{thm:W2} to this setting. Similarly as  before, one can deduce that   Theorems \ref{thm:rapport-Winfty} and  \ref{thm:rapport-W2} are also satisfied for higher dimensional Heisenberg groups.

\

\emph{Generalisation of Theorem \ref{thm:refl}:} The coupling by reflection can also be done on $\He_n$. It corresponds to the matrix $L$ defined by 
\[\left\{
\begin{array}{l}
L^{1,1}= -1\\
L^{i,i}=1 \textrm{ for }{2 \leq  i \leq n}\\ 
L^{i,j}=0  \textrm{ for } {i\neq j.}\\
\end{array}
\right.
\]

In this case, a computation easily gives that  $C_t$ and $\tilde C_t$ are independent.  Moreover since $L$ is symmetric, $J$ is also symmetric and $\sum_{i=1}^n ( J^{2i-1,2i} - J^{2i,2i-1})=0$. As a consequence, $R_t^2$ and $Z_t$ satisfy the same stochastic differential system as in the case of $\He_1$. Thus Proposition \ref{pro:assymp} holds in $\He_n$ with constants $C_p,C_p'$ and $C_p''$ independent of the dimension and Theorem \ref{thm:refl} also holds for $\He_n$.

\

\emph{Generalisation of Theorem \ref{thm:static}:} We turn to the static coupling. As before, using the symmetries  and the dilatation of the higher dimensional Heisenberg groups, it suffices to study the case $t=1$ and $a= 0$, $a'= ((x',0), (0,0) , \dots, (0,0), 0)$.
If  the vector $V=( (X_1, Y_1), \dots (X_d,Y_d), Z)$ has law $\mu_1^0$, then the vector $a'.V=(X_1+ x', Y_1), (X_2,Y_2), \dots (X_d,Y_d), Z+ \frac{1}{2} x' Y_2 )$ has law $\mu_1^{a'}$. An analogue coupling as the one in Section \ref{sec:bonus} can be defined. The horizontal part is translated by $((x,'0), \dots,(0,0))$ and conditionally on $( (X_1, Y_1), \dots (X_d,Y_d))$, we perform a coupling between the law of $Z$ and 
 the law of $Z+ \frac{1}{2} x' Y_2 $ as described in Lemma \ref{lem:sqrt}. Recall that it is adapted to the non-convex transport cost $(z,z')\mapsto \sqrt{|z-z'|}$. 

Since the heat kernel  estimate corresponding to \eqref{eq:dzp1} also holds in higher dimension (see \cite{MR2541147}), the proof finishes analogously to the one in $\He_1$. Therefore Theorem \ref{thm:static} is satisfied for higher dimensional Heisenberg groups too.

\bibliographystyle{abbrv}
\bibliography{basebib_bon_jui}
\end{document}